\documentclass{article}
\usepackage[utf8]{inputenc}
\usepackage{amsmath,amsthm,amssymb,amsfonts}
\usepackage{enumerate}
\usepackage{array,tabularx,tabulary}
\usepackage{graphicx}
\usepackage{hyperref}
\usepackage{tasks}
\usepackage{varwidth}
\usepackage{cleveref}

\usepackage{colortbl}

\title{Equitable 2-partitions of the Johnson graphs $J(n,3)$}

\author{
	Rhys~J.~Evans\\
	\small Sobolev Institute of Mathematics\\[-0.8ex]
	\small Siberian Branch of the Russian Academy of Sciences\\[-0.8ex]
	\small 4 Acad. Koptyug avenue\\[-0.8ex]
	\small 630090 Novosibirsk, Russia\\[-0.8ex]
	\small\tt rhysjevans00@gmail.com 
	\and
	Alexander~L.~Gavrilyuk\\
	\small Interdisciplinary Faculty of Science and Engineering\\[-0.8ex]
	\small Shimane University, Matsue, Japan\\[-0.8ex]
	\small\tt gavrilyuk@riko.shimane-u.ac.jp
	\and
	Sergey~Goryainov\\
	\small School of Mathematical Sciences \\[-0.8ex]
	\small Hebei Key Laboratory of Computational Mathematics and Applications\\[-0.8ex]
	\small Hebei Normal University\\[-0.8ex]
	\small Shijiazhuang  050024\\[-0.8ex]
	\small P.R. China\\[-0.8ex]
	\small\tt sergey.goryainov3@gmail.com 
	\and
	Konstantin~Vorob'ev\\
	\small Sobolev Institute of Mathematics\\[-0.8ex]
	\small Siberian Branch of the Russian Academy of Sciences\\[-0.8ex]
	\small 4 Acad. Koptyug avenue\\[-0.8ex]
	\small 630090 Novosibirsk, Russia\\[-0.4ex]
	\small Institute of Mathematics and Informatics\\[-0.8ex]
	\small Bulgarian Academy of Sciences\\[-0.8ex]
	\small 8 Acad G. Bonchev str.\\[-0.8ex]
	\small 1113 Sofia, Bulgaria\\[-0.8ex]
	\small\tt konstantin.vorobev@gmail.com\\
}

\date{\today}

\newtheorem{thm}{Theorem}
\newtheorem{lemma}[thm]{Lemma}

\newtheorem{cor}[thm]{Corollary}

\numberwithin{thm}{section}

\begin{document}

\maketitle

\begin{abstract}
	We finish the classification of equitable 2-partitions of the Johnson graphs of diameter 3, $J(n,3)$, for $n>10$. 
\end{abstract}

\section{Introduction}

The relationship between association schemes and codes was the topic of the thesis of Delsarte \cite{D_1975}. Motivated by previous authors, Delsarte puts a particular emphasis on the Hamming and Johnson schemes. He makes a comment \cite[pg 55]{D_1975}, suggesting that there does not exist any non-trivial perfect codes in the Johnson graphs.

Martin \cite{M_1992} expands on the work of Delsarte by studying completely regular subsets in detail. In his work, Martin presents the relation between perfect codes and equitable $2$-partitions more explicitly. In the literature, many substructures of regular graphs correspond to or are equivalent to equitable $2$-partitions (e.g., see \cite{N_1982}, \cite{BCGG_2019}, \cite{Ihringer}). 
Equitable $2$-partitions have been studied for several families of graphs, such as in the hypercubes by Fon-Der-Flaass \cite{F_2007}. Due to Delsarte's original work, there is continued interest in the Hamming graphs (e.g. \cite{MV_2020}) and Johnson graphs (e.g. \cite{M_2007} \cite{GG_2013}). 

Any given equitable 2-partition of a graph is naturally associated with a non-principal eigenvalue $\theta$ of the graph, and such a partition is called $\theta$-equitable. As mentioned in \cite{GG_2013}, for each integer $k\geq 2$ the $\theta$-equitable 2-partitions of $J(n,k)$ associated to the second largest and smallest eigenvalues of $J(n,k)$ have been characterised in Meyerowitz \cite{M_2003} and Martin \cite{M_1994}, respectively. In particular, for $k=3$ the third largest eigenvalue of $J(n,3)$ is only eigenvalue for which the $\theta$-equitable 2-partitions of $J(n,3)$ are not fully characterised. For $k>3$ and $n\geq2k$, all $\theta$-equitable 2-partitions of $J(n,k)$ associated to the third largest eigenvalue of $J(n,k)$ are characterised in an unpublished work of Vorob'ev \cite{V_2020}. 

In this paper, we work on the next open case of equitable 2-partitions of the graphs $J(n,3)$ associated with $\lambda_2$, the third largest eigenvalue of $J(n,3)$. In fact, we present two distinct methods to classify such equitable partitions, and characterise the $\lambda_2$-equitable 2-partitions of $J(n,3)$ for all $n>10$. After the preliminary Section \ref{sec:Jn3Def}, we present the known $\lambda_2$-equitable 2-partitions.

In Section \ref{sec:PartialDiff} we start by analysing certain eigenfunctions in the Johnson graph. The results we obtain are then used to show that for all $n\geq10$, any $\lambda_2$-equitable 2-partition is found via one of the constructions of Section \ref{sec:knJn3}.  
	
In Section \ref{sec:Jn3t2} we give an alternative approach to the classification problem. We first introduce a combinatorial tool which enables us to analyse the local structure of $\theta$-equitable 2-partitions, which has been used previously by Gavrilyuk and Goryainov \cite{GG_2013}. In Sections \ref{sec:Jt2} and \ref{sec:largen} we apply this tool to find restrictions on the local structure of a given partition. We use these results in Section \ref{sec:ng14} to show that for all $n>14$, any $\lambda_2$-equitable 2-partition is found via one of the constructions of Section \ref{sec:knJn3}.

\section{Preliminaries}\label{sec:Jn3Def}

In this section we introduce the Johnson graphs $J(n,3)$ and equitable partitions.

For a positive integer $p$, we define $\left[p\right]:=\{1,\dots,p\}$. For positive integers $p,q$, the \emph{$p\times q$-lattice} is the graph with vertex set $\{(i,j):i\in \left[p\right],j\in \left[q\right]\}$, and two distinct vertices are joined by an edge precisely when they have the same value at one coordinate.

Let $n$ be an integer, $n\geq 6$. The \emph{Johnson graph} $J(n,3)$ has vertex set \linebreak
$\{K\subseteq \left[n\right]:|K|=3\},$ 
and distinct vertices $K,L$ are adjacent if and only if \linebreak$|K\cap L|=2$. Throughout this paper, the graph $\Gamma$ will be the Johnson graph $J(n,3)$, where the value of $n$ will be specified in advance. For any triple of distinct elements $a,b,c \in \left[n\right]$, let $abc$ denote the set $\{a,b,c\}$. For distinct elements $i,j \in \left[n\right]$, denote by $ij\ast$ the set of subsets of $\left[n\right]$ of size $3$ that contain both elements $i$ and $j$. Note that $ij\ast$ induces a clique of size $(n-2)$ in $J(n,3)$.

The Johnson graph $J(n,3)$ is a distance-regular graph with diameter $3$, 
and the eigenvalues of $J(n,3)$ are
\begin{eqnarray*}
	k &=& 3(n-3),\\
	\lambda_1 &=& 2n-9,\\
	\lambda_2 &=& n-7,\\ 
	\lambda_3 &=& -3.
\end{eqnarray*}
For more information on Johnson graphs, see Brouwer, Cohen and Neumaier \cite[Section 9.1]{BCN_1989}. It is known that the neighbourhood $\Gamma(x)$ of any vertex $x$ in $J(n,3)$ is isomorphic to the $3 \times (n-3)$-lattice. In particular, there  are three maximal cliques of size $n-3$ in the
neighbourhood of $abc$, given by $ab\ast$, $ac\ast$ and $bc\ast$, and $n-3$ maximal  cliques of size $3$, given by the triples $\{abi,aci,bci\}$, where $i\in \left[n\right] \setminus \{a,b,c\}$. The $ab$-\emph{row} of $\Gamma(abc)$ is the set $ab\ast \setminus abc$. For an element $i \in \left[n\right] \setminus \{a,b,c\}$, the $i$-\emph{column} of $\Gamma(abc)$ is the set $\{abi,aci,bci\}$, and $i$ is the \emph{index} of this column.

Let $\Delta$ be a graph, and  $\Pi=\{X_{1},...,X_{q}\}$ be a partition of the vertex set $V(\Delta)$ of $\Delta$. Then the sets $X_{i}$ are called the \emph{cells} of $\Pi$, and $\Pi$ is called a \emph{$q$-partition}. Let $A_{i,j}$ be the matrix $A(\Delta)$ restricted to the rows indexed by vertices in $X_{i}$, and columns indexed by vertices in $X_{j}$. Then there is an ordering of $V(\Delta)$ such that the adjacency matrix $A(\Delta)$ has the following block matrix form: 
$$A(\Delta)=\left(\begin{array}{ccc}
A_{1,1} & \cdots & A_{1,q}\\
\vdots & \ddots & \vdots\\
A_{q,1} & \cdots & A_{q,q}
\end{array}\right)$$

Let $b_{i,j}$ be the average row-sum of $A_{i,j}$. The matrix
$$A(\Delta \slash \Pi) =\left(\begin{array}{ccc}
b_{1,1} & \cdots & b_{1,q}\\
\vdots & \ddots & \vdots\\
b_{q,1} & \cdots & b_{q,q}
\end{array}\right)$$ 
is called the \emph{quotient matrix} of $A(\Delta)$ with respect to the partition $\Pi$. The partition $\Pi$ is \emph{equitable} if for each cell $X_{i}$ and every vertex $u\in X_{i}$ we have $|X_{j}\cap \Delta(u)|=b_{i,j}$ for every $j$.   

Every equitable 2-partition in a regular graph can be naturally associated to an eigenvalue of the graph.

\begin{lemma}\label{lem:theta}
	Let $\Pi$ be an equitable $2$-partition of a $k$-regular graph $\Delta$ with quotient matrix
	$$A(\Delta \slash \Pi) =\left(\begin{array}{cc}
	b_{1,1} & b_{1,2}\\
	b_{2,1} & b_{2,2}
	\end{array}\right)$$ 
	Then the eigenvalues of $A(\Delta \slash \Pi)$ are given by $k = b_{1,1} + b_{1,2} = b_{2,1} + b_{2,2}$ and $\theta = b_{1,1}-b_{2,1} = b_{2,2}-b_{1,2}$, where $\theta$ is an eigenvalue of $\Delta$, $\theta\not= k$.
\end{lemma}
\begin{proof}
	This is a simple application of eigenvalue interlacing of quotient matrices and a routine calculation of the eigenvalues of a 2x2 matrix (see, e.g., \cite{GG_2013}).
\end{proof}

Let $\Delta$ be a $k$-regular graph and $\theta$ be a real number, $\theta\not=k$. An equitable $2$-partition $\Pi$ of $\Delta$ is \emph{$\theta$-equitable} if $A(\Delta \slash \Pi)$ has eigenvalue $\theta$ . By Lemma \ref{lem:theta}, the quotient matrix of an equitable $2$-partition $\Pi$ has eigenvalues $k$ and $\theta$, where $\theta$ is an eigenvalue of $\Delta$. Therefore, we can enumerate equitable 2-partitions of a regular graph by enumerating $\theta$-equitable $2$-partitions for each eigenvalue $\theta$ of $\Gamma$.   

\section{Known equitable 2-partitions in $J(n,3)$}\label{sec:knJn3}

The $\theta$-equitable $2$-partitions of $J(n,3)$ have been classified for the eigenvalues $\theta=\lambda_1$ and $\lambda_3$. For more references and information on these partitions, see \cite{GG_2013}. The open case corresponds to the eigenvalue $\lambda_2$.

An ad hoc analysis of the classification problem for small values of $n$ can be found in  Mogilnykh \cite{M_2007} and Avgustinovich and Mogilynkh \cite{AM_2008,AM_2011}. For example, they find all of the possible quotient matrices for an equitable 2-partition in $J(8,3)$. Avgustinovich and Mogilynkh give many constructions, but there were no full classification results for $\lambda_2$-equitable 2-partitions in $J(n,3)$ for any $n>6$ before the current paper.

In \cite{AM_2011}, an equitable $3$-partition of $J(2m,3)$ was constructed for all $m\geq 3$. This construction was used to produce three families of $\lambda_2$-equitable $2$-partitions of $J(2m,3)$. Here we give a detailed presentation of this construction.

Let $U = \{u_1,\dots,u_m\}$ and 
$W = \{w_1,\dots,w_m\}$ be sets of integers such that $U \cup W = \left[2m\right]$ (i.e. $U$ and $W$ partition the set $\left[2m\right]$). Let $\Delta_{\{U,W\}}$ be the graph with vertices $U\cup W$ and edge set $\{u_iw_j:i\not=j\}$. In other words, $\Delta_{\{U,W\}}$  is constructed by taking the complete bipartite graph with parts $U,W$, and then removing the edges  $u_1w_1,u_2w_2,\dots,u_mw_m$.  
\begin{figure}[ht!]
	\centering
	\includegraphics[scale=0.30]{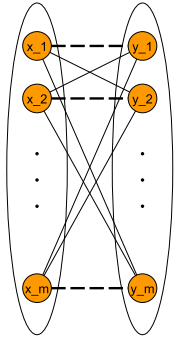}
	\caption[The graph $\Delta_{\{U,W\}}$]{The graph $\Delta_{\{U,W\}}$.}	
\end{figure}

There are three ``types'' of unordered triples of vertices in $\Delta_{\{U,W\}}$, each of which are illustrated in Figure \ref{fig:delUW}.
\begin{figure}
	\centering
	\includegraphics[scale=0.30]{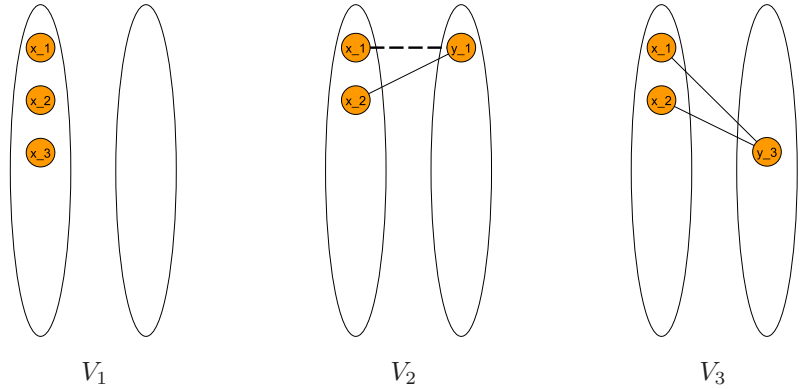}
	
	$V_1$~~~~~~~~~~~~~~~~~~~~~~~~~~~~~~~~$V_2$~~~~~~~~~~~~~~~~~~~~~~~~~~~~~~~~$V_3$				
	\caption[Triples of vertices in $\Delta_{\{U,W\}}$]{Triples of vertices in $\Delta_{\{U,W\}}$}\label{fig:delUW}
\end{figure}
Any set of three distinct vertices $abc \subseteq \left[2m\right]$ lies in one of the following sets:
\begin{eqnarray*}
	X_1&=&\{u_iu_ju_k:i,j,k\text{ distinct} \}\cup\{w_iw_jw_k:i,j,k\text{ distinct}\}.\\
	X_2&=&\{u_iu_jw_i:i,j\text{ distinct}\} \cup \{w_iw_ju_i:i,j\text{ distinct}\}.\\
	X_3&=&\{u_iu_jw_k:i,j,k \text{ distinct}\}\cup\{w_iw_ju_k:i,j,k \text{ distinct} \}.
\end{eqnarray*}
Now we regard the triples of vertices of $\Delta_{\{U,W\}}$ as the vertices of $J(2m,3)$. The partition $\Pi=\{X_1,X_2,X_3\}$ gives a 3-partition of the vertex set of $J(2m,3)$.

Consider $\Gamma:=J(2m,3)$ and vertices $t_1=u_1u_2u_3,t_2=u_1u_2w_1,t_3=u_1u_2w_3$, so $t_i \in X_i$ for each $i$. The following $3\times(2m-3)$-arrays represent the $3\times(2m-3)$-grids induced by $\Gamma(t_1)$, $\Gamma(t_2)$ and $\Gamma(t_3)$. The indexes of columns are given below the braces. The rows are given by the comments on the right. The entries are from the set $\{1,2,3\}$, and an entry equals to $j$ whenever the corresponding vertex belongs to the cell $X_j$.

\begin{center}
	$
	\begin{array}{c@{}c@{}@{}c@{}@{}c@{}@{}c@{}@{}c@{}@{}c@{}}
	\Gamma(t_1): ~~&
	~
	\underbrace{\begin{array}{@{}c@{}@{}c@{}@{}c@{}}  
		1 & ~~\ldots~~ & 1 \\
		1 & ~~\ldots~~ & 1 \\
		1 & ~~\ldots~~ & 1
		\end{array}}_{U\setminus \{u_1,u_2,u_3\}}
	& ~
	\underbrace{\begin{array}{@{}c@{}@{}c@{}@{}c@{}}  
		3 & ~~\ldots~~ & 3 \\
		3 & ~~\ldots~~ & 3 \\
		3 & ~~\ldots~~ & 3
		\end{array}}_{W\setminus \{w_1,w_2,w_3\}}
	& 
	\underbrace{\begin{array}{@{}c@{}}  
		2 \\								
		2 \\
		3 
		\end{array}}_{w_1}
	& ~
	\underbrace{\begin{array}{@{}c@{}}  
		2 \\
		3 \\
		2 
		\end{array}}_{w_2}
	& ~
	\underbrace{\begin{array}{@{}c@{}}  
		3 \\
		2 \\
		2 
		\end{array}}_{w_3}
	& ~
	\begin{array}{@{}c@{}}  
	~\leftarrow\text{$u_1u_2$-row} \\
	~\leftarrow\text{$u_1u_3$-row} \\
	~\leftarrow\text{$u_2u_3$-row} 
	\end{array}
	\end{array}
	\bigskip
	$
	$
	\begin{array}{c@{}c@{}@{}c@{}@{}c@{}@{}c@{}@{}c@{}@{}c@{}}
	\Gamma(t_2): ~~&
	~
	\underbrace{\begin{array}{@{}c@{}@{}c@{}@{}c@{}}  
		1 & ~~\ldots~~ & 1 \\
		2 & ~~\ldots~~ & 2 \\
		3 & ~~\ldots~~ & 3
		\end{array} }_{U\setminus \{u_1,u_2\}}
	& ~
	\underbrace{\begin{array}{@{}c@{}@{}c@{}@{}c@{}}  
		3 & ~~\ldots~~ & 3 \\
		2 & ~~\ldots~~ & 2 \\
		3 & ~~\ldots~~ & 3
		\end{array}}_{W\setminus \{w_1,w_2\}}
	& ~
	\underbrace{\begin{array}{@{}c@{}}  
		2 \\
		2 \\
		2 
		\end{array}}_{w_2}
	& ~
	\begin{array}{@{}c@{}}  
	~\leftarrow\text{$u_1u_2$-row} \\
	~\leftarrow\text{$u_1w_1$-row} \\
	~\leftarrow\text{$u_2w_1$-row} 
	\end{array}
	\end{array}
	\bigskip
	$
	$
	\begin{array}{c@{}c@{}@{}c@{}@{}c@{}@{}c@{}@{}c@{}@{}c@{}}
	\Gamma(t_3): ~~&
	~
	\underbrace{\begin{array}{@{}c@{}@{}c@{}@{}c@{}}  
		1 & ~~\ldots~~ & 1 \\
		3 & ~~\ldots~~ & 3 \\
		3 & ~~\ldots~~ & 3
		\end{array} }_{U\setminus \{u_1,u_2,u_3\}}
	& ~
	\underbrace{\begin{array}{@{}c@{}@{}c@{}@{}c@{}}  
		3 & ~~\ldots~~ & 3 \\
		3 & ~~\ldots~~ & 3 \\
		3 & ~~\ldots~~ & 3
		\end{array}}_{W\setminus \{w_1,w_2,w_3\}}
	& 
	\underbrace{\begin{array}{@{}c@{}}  
		1 \\
		2 \\
		2 
		\end{array}}_{u_3}
	& ~
	\underbrace{\begin{array}{@{}c@{}}  
		2 \\
		2 \\
		3 
		\end{array}}_{w_1}
	& ~
	\underbrace{\begin{array}{@{}c@{}}  
		2 \\
		3 \\
		2 
		\end{array}}_{w_2}
	& ~
	\begin{array}{@{}c@{}}  
	\leftarrow\text{$u_1u_2$-row} \\
	~\leftarrow\text{$u_1w_3$-row} \\
	~\leftarrow\text{$u_2w_3$-row} 
	\end{array}									\end{array}
	$
\end{center}

With the knowledge of these neighbourhoods and the symmetry of the graph $\Delta_{\{U,W\}}$, we deduce that $\Pi=\{X_1,X_2,X_3\}$ is an equitable $3$-partition of $J(2m,3)$.

\begin{lemma}								Let $\Pi=\{X_1,X_2,X_3\}$ be the partition of the vertices of $\Gamma=J(2m,3)$ defined above. Then $\Pi$ is equitable, and has quotient matrix 
	$$
	A(\Gamma / \Pi) =
	\left(
	\begin{array}{ccc}
	3m-9 & 6 & 3m-6  \\
	m-2 & 2m-1& 3m-6 \\
	m-2 & 6 & 5m-13
	\end{array}
	\right).
	$$
\end{lemma}
\begin{proof}
	Any vertex $abc\in \Gamma$ lies in $X_j$ for exactly one value $j\in \{1,2,3\}$. Then there is a permutation of $\left[2m\right]$ such that the $3\times (2m-3)$-array of $\Gamma(abc)$ is the array of $t_j$ in the above. This observation shows that the partition $\Pi$ is equitable, and we can use the arrays of the neighbourhoods to determine the quotient matrix.
\end{proof}

Now we construct three different equitable $2$-partitions by merging the cells of the partition $\Pi$. Let $\Pi_1 := \{X_2 \cup X_3,X_1\}$, $\Pi_2 := \{X_1 \cup X_3,X_2\}$ and $\Pi_3 = \{X_3,X_1 \cup X_2\}$. 

\begin{lemma}\label{thm:knJn3}
	Let $\Pi_j$ be the partitions of the vertices of $\Gamma=J(2m,3)$ defined above. Then the partitions $\Pi_j$ are equitable, and have quotient matrices 
	$$
	A(\Gamma / \Pi_1) = 
	\left(
	\begin{array}{cc}
	5m - 7 & m-2 \\
	3m & 3m-9
	\end{array}
	\right),
	$$
	$$
	A(\Gamma / \Pi_2) = 
	\left(
	\begin{array}{cc}
	6m - 15 & 6 \\
	4m-8 & 2m-1
	\end{array}
	\right),$$
	$$
	A(\Gamma / \Pi_3) = 
	\left(
	\begin{array}{cc}
	5m - 13 & m+4\\
	3m-6 & 3m-3
	\end{array}
	\right).$$
\end{lemma}
\begin{proof}
	This follows from \cite[Lemma 1]{AM_2011}, which shows the quotient matrix $A(\Gamma / \Pi)$ has equal non-diagonal entries for each column. Thus, we can merge cells of the equitable 3-partition to get equitable 2-partitions. The quotient matrices also follow easily from this argument. 
\end{proof}

\section{Proof based on partial differences}\label{sec:PartialDiff}

In this section, we prove that for $n>10$, all $\lambda_2$-equitable $2$-partitions have the same structure and quotient matrices as in Lemma \ref{thm:knJn3}. 

Given a graph $\Delta$, a real-valued function $f\colon V(\Delta)\longrightarrow{\mathbb{R}}$ is called a {\em $\theta$--eigenfunction} of $G$ if the equality $$\theta\cdot f(x)=\sum_{y\in V(\Delta(x))}f(y)$$ \\
holds for any $x\in V(\Delta)$ and $f$ is not the all-zeros function. Equivalently, the vector of values of a $\theta$--eigenfunction is an eigenvector of the adjacency matrix of $\Delta$ with an eigenvalue $\theta$. 

Given a real-valued function $f$ defined on vertices of $J(n,w)$ and distinct $a,b\in [n]$, define a {\it partial difference $f_{ab}$} as follows. For every $(w-1)$-subset $y$ of $[n]\setminus \{a,b\}$, set $$f_{ab}(y)=f(y\cup \{a\})-f(y\cup \{b\}).$$
Clearly, $f_{ab}$ can be treated as a function defined on vertices of $J(n-2,w-1)$. Moreover, there is a useful connection between a function and its partial difference. Recall the spectra of the Johnson graph $J(n,w)$ consists of $\lambda_{i}(n,w)=(w-i)(n-w-i)-i$, $i\in \{0,1, \dots, w\}$ (see \cite[Section~9.1]{BCN_1989}).

\begin{lemma}\label{lem:partdif}(\cite{VMV_2018})
	Let $f$ be a $\lambda_{i}(n,w)$-eigenfunction of $J(n,w)$, $a$, $b$ $\in [n]$. Then $f_{ab}$ is a
	$\lambda_{i-1}(n-2,w-1)$-eigenfunction of $J(n-2,w-1)$ or the all-zeros function.  
\end{lemma}

The following lemma shows that all-zeros partial differences impose additional restrictions on the values of the initial function.  

\begin{lemma}\label{L:zero_part_diff}(\cite{VMV_2018})
	Let $f$ be a real-valued function defined on vertices of $J(n,w)$. Suppose that $f_{ab}\equiv 0$ and $f_{ac}\equiv 0$ for some pairwise distinct $a,b,c\in [n]$. Then $f_{bc}\equiv 0$.
\end{lemma}


Throughout the rest of this paper, we will let $X = \{X_1,X_2\}$, a $\theta$-equitable $2$-partition of $J(n,3)$. 
Let $\theta$ be a non-principal eigenvalue of $\Gamma=J(n,3)$ and suppose $X$ is a $\theta$-equitable $2$-partition with quotient matrix 
$$
A(\Gamma / X) =
\left(
\begin{array}{cc}
p_{11} & p_{12} \\
p_{21} & p_{22} \\
\end{array}
\right),
$$
where $p_{12}<p_{21}$ and  $p_{12}+p_{21}=2n-2$.

For a vertex $u\in V(\Gamma)$, the $X_1$\emph{-indicator function} on $u$ is
$$
\overline{u} :=
\left\{
\begin{array}{ll}
1 & \text{if }u \in X_1; \\
0 & \text{if }u \in X_2,
\end{array}
\right.
$$
and for a set of vertices $U\subseteq V(\Gamma)$ we define \begin{equation*}
\overline{U}:=\sum\limits_{u \in U}\overline{u}
\end{equation*}

Consider a function $f\colon V(\Gamma)\to \mathbb{R}$ given by  
$$f(x)= \begin{cases}
\frac{-p_{21}}{p_{12}+p_{21}}, x \in X_1\\
\frac{p_{12}}{p_{12}+p_{21}}, x \in X_2.
\end{cases}$$
By direct calculations one can easily check that $f$ is a $\lambda_2$-eigenfunction of $\Gamma$. 

The following proof of the main result is based on an analysis of the structure of partial differences of $f$, and we need one more auxiliary statement for it.

\begin{lemma}\label{L:Derivates}
	Let $f$ be $\lambda_1(n,2)$-eigenfunction of $J(n,2)$ taking values $\{-1,0,1\}$, $f\not\equiv 0$, $n\geq 5$. Then $f$ is equal to one of the following functions:
	\begin{enumerate}
		\item $f_1(x)= \begin{cases}
		1, &a\in x,\,b\notin x\\
		-1, &a\notin x,\,b\in x\\
		0,&\text{otherwise,} \end{cases}$ \\for some $a,b\in [n]$,

		\item $f_2(x)= \begin{cases}
		1, &x \subset M_1,\\
		-1, &x \subset M_2,\\
		0,&\text{otherwise,} \end{cases}$\\ for some $M_1\cup M_2=[n], |M_1|=|M_2|=\frac{n}{2}$, and even $n$.
		
	\end{enumerate}
\end{lemma}
\begin{proof}
	The main idea is to use one-to-one correspondence between eigenspaces of different Johnson graphs. In particular, it is known (see, for example \cite{D_1975}), that every $\lambda_i(n,w)$-eigenfunction $h_1$ of the graph $J(n,w)$, $1 \leq i\leq w-1$ can be obtain from some $\lambda_i(n,i)$-eigenfunction $h_2$ of $J(n,i)$ by so-called inducing operator $$h_1(x)=\sum_{y\in V(J(n,i)), y\subseteq x}{h_2(y)},$$ where $V(J(n,i))$ is the set of vertices of $J(n,i)$. 
	
	{In our case it means that $f$ can be obtained from some $\lambda_1(n,1)$-eigenfunction $g$ of the graph $J(n,1)$. The function $g$ can be treated as one defined on numbers $1,2,\dots, n$, and $f$, in turn, as one defined on pairs of these numbers.}
	
	Let us denote $\gamma_i=g(i)$, consequently, for any distinct $a,b\in [n]$ we have $f(ab)=\gamma_{a}+\gamma_{b}$. Since $g$ is a $\lambda_1(n,1)$-eigenfunction it must be orthogonal to a constant function. Therefore, $\gamma_1+\gamma_2+\dots+\gamma_n=0$. 
	The next step is the determination of all sets $A=\{\gamma_1,\gamma_2,\dots \gamma_n\}$ such that the set $\{\gamma_i+\gamma_j|\, i,j\in [n],\, i\neq j\}$ does not contain any elements except $-1$, $0$ and $1$.
	
	Let $A$ contain at least four pairwise distinct values, for example $\gamma_1>\gamma_2>\gamma_3>\gamma_4$. It means that $\gamma_1+\gamma_2>\gamma_1+\gamma_3>\gamma_1+\gamma_4>\gamma_3+\gamma_4$ and $f$ takes more than three values. 
	
	Let $A$ contain exactly three pairwise distinct values $\gamma_1>\gamma_2>\gamma_3$. Consequently, $\gamma_1+\gamma_2=1$, $\gamma_1+\gamma_3=0$ and $\gamma_2+\gamma_3=-1$. Obviously, $\gamma_1>0$ and $\gamma_3<0$. It is easy to see that $A$ must contain only one element $\gamma_1$ and one $\gamma_3$. Otherwise, $f$ will take one more new value $2\gamma_1$ or $2\gamma_3$. So, $A$ consists of one $\gamma_1$, one $\gamma_3$ and $n-2$ elements $\gamma_2$. If $\gamma_2\neq 0$ we again find a new value $2\gamma_2$ of $f$. Finally, we conclude, that $\gamma_1=1$, $\gamma_2=0$ and $\gamma_3=-1$ and build the function $f_1$.
	
	The last possible case is when $A$ contains exactly two pairwise distinct values $\gamma_1$, $\gamma_2$, $\gamma_1>\gamma_2$. Let $n_1$ and $n_2$ be multiplicities of these elements in $A$. By our previous arguments we have that $n_1\gamma_1+n_2\gamma_2=0$ and $n_1+n_2=n$. Clearly, if $n_1=1$ then $f$ takes values $\gamma_1\frac{n-2}{n-1}$ and $-\gamma_1\frac{2}{n-1}$. For $n\geq 5$, these values cannot be equal to $1$ and $-1$. Similarly, $n_2\neq 1$ too. Hence, $f$ takes three values $2\gamma_1$, $\gamma_1+\gamma_2$ and $2\gamma_2$. Particularly, $\gamma_1+\gamma_2=0$. Therefore, $n_1=n_2$, $\gamma_1=\frac{1}{2}$, $\gamma_2=-\frac{1}{2}$. Finally, we construct the function $f_2$.    
\end{proof}

In the rest of the paper, functions equal to $f_1$ and $f_2$ will be called functions of type $1$ and $2$ respectively. We will say, that elements $a, b$  and sets $M_1, M_2$ are {\it defining} for the functions $f_1$ and $f_2$ respectively. The knowledge of structure of possible partial differences allows us to prove the following statement.

\begin{lemma}\label{L:Supports}
	Let $X=(X_1, X_2)$ be $\lambda_2$-equitable $2$-partition of $J(n,3)$, $n\geq 7$, with the quotient matrix $\left(\begin{array}{cc}
	p_{11} & p_{12}\\
	p_{21} & p_{22}
	\end{array}\right)$ and $g$ be a characteristic function of $X_1$. Suppose that the set $$\{g_{ij}|\, 1\leq i<j\leq n\}$$ contains $t_1$ and $t_2$ partial differences of types $1$ and $2$ correspondingly, and ${n \choose 2}-t_1-t_2$ all-zero differences. Then $$p_{12}p_{21}n(n-2)=24t_1(n-4)+3t_2(n-2)(n-4).$$ 
\end{lemma}
\begin{proof}
	It is easy to see that there is a one-to-one correspondence between pairs of neighbours from distinct cells of the partition and nonzero values of partial differences of $g$. Clearly, the number of such pairs equals $\frac{p_{12}p_{21}}{p_{12}+p_{21}}{n \choose 3}$. By the structure of our partial differences we conclude that the number of their nonzero values is equal to $2t_1(n-4)+2t_2{\frac{n-2}{2} \choose 2}$. After equating these expressions and using $p_{12}+p_{21}=2n-2$ we get the claim of the theorem.  
\end{proof}

Based on our auxiliary statements we are ready to prove the main result of the paper.

\begin{thm}\label{T:PartDerMain} Let $X=(X_1, X_2)$ be $\lambda_2$-equitable $2$-partition of $J(n,3)$, $n\geq 12$. Then the partition $X$ can be constructed from a matching on the elements of $\left[n\right]$, which corresponds to an instance of a partition $\Pi_1$, $\Pi_2$ or $\Pi_3$.
\end{thm}

\begin{proof}
	As it mentioned before, cases of odd $n$ and $p_{12}=p_{21}$ were characterized before by Gavrilyuk and Goryainov \cite{GG_2013}. However, the following proof allows to cover these cases too, so we shall not exclude these cases.
	Consider $$f(x)= \begin{cases}
	\frac{-p_{21}}{p_{12}+p_{21}}, x \in X_1\\
	\frac{p_{12}}{p_{12}+p_{21}}, x \in X_2.
	\end{cases}$$
	As it was previously established, $f$ is $\lambda_2$-eigenfunction of $J(n,3)$, where $\lambda_2=\lambda_2(n,3)$. By definition of $f$ and Lemma \ref{lem:partdif}, every partial difference of $f$ is a $\lambda_1(n-2,2)$-eigenfunction of $J(n-2,2)$ taking three values $\{-1,0,1\}$ or the identically zero function. This fact allows to use Lemma \ref{L:Derivates} for these partial differences in our arguments.
	
	Suppose that there are no partial differences of $f$ of type $1$. Hence, $n$ must be even.
	Let us take some not identically zero partial difference of $f$. Without loss of generality, we take $f_{12}$ with defining sets $M_1=\{3,4,5,\dots, \frac{n}{2}+1\}$, $M_2=\{\frac{n}{2}+2,\frac{n}{2}+3,\dots, n\}$.
	By definition of $f_{12}$ we conclude that $$1ab, 2cd \in X_2 \text{ and } 2ab, 1cd \in X_1$$ for every $a,b\in M_1$ and $c,d\in M_2$.
	
	Suppose that for some  $a,b\in M_1$ we have $f_{ab}\not\equiv 0$. We can assume that $a=3$ and $b=4$. By the structure of $f_{12}$ we know that $f_{34}(1i)=f_{34}(2i)=0$ for $i\in \{5,6,\dots \frac{n}{2}+1\}$. Hence, elements $1$ and $2$ belong to one defining set of $f_{34}$ and $5,6,\dots \frac{n}{2}+1$ to other one. Without loss of generality we may assume that $f_{34}(12)=1$ (if it equals $-1$ we swap elements $3$ and $4$) and $123\in X_2$. Consequently, $f_{14}(23)=1$. From our arguments follows that $f_{14}(ij)=0$ for $i,j\in \{5,6,\dots \frac{n}{2}+1\}$. However, for $n\geq 12$ the set $\{5,6,\dots \frac{n}{2}+1\}$ contains at least $3$ elements, so at least one pair $(i',j')$ from a one of defining sets of $f_{14}$. Therefore, $f_{14}(i'j')\neq 0$ and we get a contradiction.
	
	By similar arguments for all pairs $(i,j)$ such that $i,j\in M_1$ or $i,j\in M_2$ we have that $f_{ij}\equiv 0$.
	By the structure of $f_{12}$ and previous arguments we know that $f_{13}(ij)=0$ for $i,j\in \{4,5,\dots \frac{n}{2}+1\}$. For $n\geq 10$ the set $\{4,5,\dots \frac{n}{2}+1\}$ contains not less than $3$ elements, so $f_{13}$ has no defining sets and $f_{13}\equiv 0$. By similar arguments $f_{1i}\equiv 0$ for $i\in M_1$ and $f_{2i}\equiv 0$ for $i\in M_2$. 
	It guarantees us that each of the sets 
	$$S_1=\{ijk|\, i,j,k\in M_1\cup\{1\}, i\neq j, j\neq k, i\neq k\},$$ 
	$$S_2=\{ijk|\, i,j,k\in M_2\cup\{2\}, i\neq j, j\neq k, i\neq k\},$$ 
	$$S_3=\{ijk|\, i,j\in M_1\cup\{1\}, k\in M_2\cup\{2\}, i\neq j\},$$ 
	$$S_4=\{ijk|\, i,j\in M_2\cup\{2\}, k\in M_1\cup\{1\}, i\neq j\}$$ is a subset of $X_1$ or $X_2$.
	
	The triple $123\in S_1$ has $\frac{3}{2}n-9$ neighbours in $S_1$ and $\frac{3}{2}n$ ones in $S_3$. Since all triples in $S_1$ are elements of a one cell of the partition $X$, we conclude that $p_{12}=\frac{3}{2}n$ or $p_{21}=\frac{3}{2}n$. By our agreement $p_{21}\geq p_{12}$, so the quotient matrix of the partition equals 	
	$$\left(
	\begin{array}{cc}
	\frac{3}{2}n-7 & \frac{n}{2}-2 \\
	\frac{3}{2}n & \frac{3}{2}n-9 \\
	\end{array}
	\right),$$
	$S_1\subseteq X_2$, $S_3\subseteq X_1$. By similar arguments we have that $S_2\subseteq X_2$, $S_4\subseteq X_1$ and the partition is an instance of $\Pi_1$.

	The remaining case is if there is a partial difference of $f$ of type $1$.\\
	After appropriate rearrangement of elements of $[n]$ we may consider $f_{12}\not\equiv 0$ with defining elements $3$ and $4$. In other words, $$13i,\, 24i \in X_2 \text{ and } 14i,\, 23i  \in X_1 \text{ for } i \in \{5,6,\dots n\},$$
	and $$\overline{1ab}=\overline{2ab},\, \overline{3ab}=\overline{4ab} \text{ for } a,b\in  \{5,6,\dots n\},$$  $$\overline{123}=\overline{124},\, \overline{134}=\overline{234}.$$
	
	Values $\overline{123}, \overline{124}, \overline{134}, \overline{234}$ play a crucial role in the determining structure of the partition $X$, and we need to consider all admissible options.  
	\begin{enumerate}
		\item $\overline{123}=\overline{124}=1, \overline{134}=\overline{234}=0$.\\
		Equivalently, $123, \,124\in X_1$ and $134,\, 234 \in X_2$. By direct calculations, $123$ already has $n-2$ neighbours from $X_2$, so $p_{12}\geq n-2$. As we know $p_{12}+p_{21}=2n-2$, $p_{12}\leq p_{21}$. 
		Suppose that $p_{12}=n-2$, $p_{21}=n$. Consequently remaining neighbours of $123$ belong to $X_1$, i.e. $12i\in X_1$, $i\in \{5,6,\dots, n\}$. 
		
		Consider the function $f_{14}$. By our arguments $f_{14}(2i)=-1$ for $i\in \{5,6,\dots,n\}$. Therefore, $f_{14}$ is of type $1$ with defining elements $2$ and $3$. Hence, $f_{14}(23)=0$, but $123\in X_1$ and $234\in X_2$. We get a contradiction.
		
		The last case here is $p_{12}=p_{21}=n-1$. We need to put one more neighbour of $123$ to $X_2$. Without loss of generality we have $125\in X_2$ and $12i\in X_1$, $i\in \{6,7,\dots, n\}$. By our arguments $f_{14}(2i)=-1$ for $i\in \{3,6,7,\dots,n\}$, and $f_{14}(25)=0$. Consequently, $f_{14}$ is of type $1$ with defining elements $2$ and $5$. Therefore, $15i\in X_2$ and $45i\in X_1$ for $i\in \{3,6,7,\dots,n\}$. In particular, $345\in X_1$. The vertex $134 \in X_2$ already has $n-1$ neighbours from $X_1$, consequently, $34i\in X_2$ for $i\in \{1,2,6,7,\dots,n\}$. Now $345\in X_1$ has $n-1$ neighbours from $X_2$, consequently, $35i\in X_1$ for $i\in \{2,4,6,7,\dots,n\}$. Finally, $125\in X_2$ has $n-1$ neighbours from $X_1$, consequently, $25i\in X_12$ for $i\in \{1,4,6,7,\dots,n\}$.
		
		Putting it all together we have that
		$$123,\, 124,\, 145,\, 235,\, 345,\,  14i,\, 23i,\, 12i,\,35i,\,45i \in X_1,$$ 
		$$134,\, 234,\,135,\, 245,\,124,\,13i,\, 24i,\, 34i,\,15i,\,25i \in X_2$$
		for  $i\in \{6,7,\dots,n\}$.
		By direct calculation it is easy to see that all partial differences $f_{ij}$, $1\leq i<j\leq 5$ are of type $1$. 
		
		Consider a difference $f_{67}$. We know that $f_{67}(ij)=0$ for $i\in \{6,7,\dots,n\}$. Therefore $f_{67}$ is of type $2$ or the all-zero function. Consider a function $g(x)=\frac{p_{12}}{p_{12}+p_{21}}-f(x)$. Clearly, $g$ is the characteristic function of $X_1$ and corresponding partial differences of $f$ and $g$ always have the same type.
		As it was established above $g$ has at most $5(n-5)$ partial differences of type $2$. 
		Since for $n> 10$ we have that $2(n-4)<2{\frac{n-2}{2} \choose 2}$, the maximum possible sum of numbers of nonzero values of partial differences of $g$ is in the case when we have $5(n-5)$ differences of type $2$ and ${n \choose 2}-5(n-5)$ ones of type $1$. Therefore, by Lemma \ref{L:Supports} we have that $$(n-1)^2n(n-2)\leq 24\cdot \left({n \choose 2}-5(n-5)\right)(n-4)+15(n-5)(n-2)(n-4),$$
		which is equivalent to $$(n-10)^2(n-6)(n-5)\leq 0.$$
		This inequality does not take place for $n>10$ and we get a contradiction
		
		\item $\overline{123}=\overline{124}=\overline{134}=\overline{234}=1$.\\
		In other words,  $123, \,124, \, 134,\, 234 \in X_1$. Each of these vertices has $n-4$ neighbours from $X_2$, so $p_{12}\geq n-4$. There are four admissible quotient matrices:
		$$\left(
		\begin{array}{cc}
		2n-8 & n-1 \\
		n-1 & 2n-8 \\
		\end{array}
		\right),
		\left(
		\begin{array}{cc}
		2n-7 & n-2 \\
		n & 2n-9 \\
		\end{array}
		\right),$$
		$$
		\left(
		\begin{array}{cc}
		2n-6 & n-3 \\
		n+1 & 2n-10 \\
		\end{array}
		\right),
		\left(
		\begin{array}{cc}
		2n-5 & n-4 \\
		n+2 & 2n-11 \\
		\end{array}
		\right).
		$$ 
		\begin{enumerate}
			\item $p_{12}=n-1$.
			
			Take the vertex $123\in X_1$. It has $n-4$ neighbours from $X_2$, so we need to find three more. Without loss of generality let us take $125,\,126,\,127\in X_2$ and $12i\in X_1$ for $i\in \{8,9,\dots, n\}$. Next, consider $f_{13}$. By our arguments we know that $f_{13}(25)=f_{13}(26)=f_{13}(27)=1$, $f_{13}(24)=0$ and $f_{13}(2i)=0$ for $i\in \{8,9,\dots, n\}$. If $f_{13}$ is of type $2$ then $M_1=\{4,5,6,7\}$ and $|M_1|=\frac{n}{2}-1$, but it contradicts to $n>10$. Hence, it is of type $1$, and $2$ must be a defining element. However, $f_{13}(28)=f_{13}(29)=0$, so this case is also not possible.
			
			\item $p_{12}=n-2$.
			
			Here we repeat arguments from the previous case. As a result we have that $125,\,126\in X_2$ and $12i\in X_1$ for $i\in \{7,8,\dots, n\}$. Next,  $f_{13}(25)=f_{13}(26)=1$, $f_{13}(24)=0$ and $f_{13}(2i)=0$ for $i\in \{7,8,\dots, n\}$. Again if $f_{13}$ is of type $2$  then $M_1=\{4,5,6\}$ and $|M_1|=\frac{n}{2}-1$, but it contradicts to $n>10$. Therefore, it is of type $1$, and $2$ is a defining element. The fact that $f_{13}(27)=f_{13}(28)=0$ lead as to a contradiction.
			\item $p_{12}=n-3$.
			
			Consider the vertex $123\in X_1$. It has $n-4$ neighbours from $X_2$, so we need to find one more. Without loss of generality let us take $125\in X_2$ and $12i\in X_1$ for $i\in \{6,7,\dots, n\}$. Again let us analyse $f_{13}$. By construction we have $f_{13}(25)=1$ and $f_{13}(2i)=0$ for $i\in \{6,7,\dots, n\}$.
			If $f_{13}$ is of type $2$ then the defining set containing elements $2$, $5$ may contain the element $4$ and nothing else. Hence, for $n>8$ it is not realizable. Consequently, it is of type $1$. Since $f_{13}(25)=1$, $2$ or $5$ is a defining element. In the first case it follows from $f_{13}(24)=0$ that $4$ must be second defining element. Therefore, $f_{13}(26)=1$. We get a contradiction. The last possible case is when $5$ is one defining element and the second one $i$ belongs to $\{4,6,7,\dots, n\}$. For all these $i$, we have $f_{13}(5i)=0$ and consequently there is no such an element.   
			\item $p_{12}=n-4$.
			
			This quotient matrix immediately guarantees us that for $i\in \{5,6,\dots, n\}$, we have $12i, 34i\in X_1$. 
			Consequently, $f_{13}(4i)=f_{13}(2i)=0$ for $i\in \{5,6,\dots, n\}$ and $f$ cannot be of type $1$ or of type $2$. Hence, $f_{13}\equiv 0$. Similarly, $f_{24}\equiv 0$.
			
			By the structure of $f_{12}$ and $f_{34}$ we know that $f_{12}(ij)=0$ and $f_{34}(ij)=0$ for $i,j\in \{5,6,\dots, n\}$. 
			
			Therefore, we conclude that for any $i,j\in \{5,6,\dots, n\}$ following equalities take place  $$f(1ij)=f(2ij)=f(3ij)=f(4ij).$$
			
			Suppose that for all $i,j\in \{5,6,\dots, n\}$ we have $f_{ij}\equiv 0$. Hence, all triples from $\{5,6,\dots, n\}$ belong to one cell of the partition. The same is true for triples $\{kij| k\in \{1,2,3,4\}\,i,j\in \{5,6,\dots, n\} \}$.
			Since our partition is equitable, these two cells must be different. Therefore, the vertex $567$ has $3(n-7)$ neighbours of it's cell and $12$ neighbours from other one. Therefore, $p_{21}=12$ or $p_{12}=12$. The first case leads us to $n=10$ , but $n\geq 12$. So, the last opportunity is $n=16$ and $\{kij| k\in \{1,2,3,4\}\,i,j\in \{5,6,\dots, n\} \}\subseteq X_2$. In this case the vertex $125\in X_1$ has exactly $2+2(n-5)=24$ neighbours from $X_2$ instead of $n+2=18$ and we get a contradiction.
			
			Without loss of generality we may consider $f_{56}\not\equiv 0$. We know that $$f_{56}(12)=f_{56}(13)=f_{56}(14)=f_{56}(23)=f_{56}(24)=f_{56}(34)=0.$$
			Hence, $f$ is of type $1$ and we may take elements $7$ and $8$ as it's defining elements. By similar arguments as we provided above we conclude that for $i\in \{1,2,3,4,9,10,\dots, n\}$, we have $57i, 68i, 58i, 67i \in X_1$ and $56i, 78i\in X_2$. 
			All these conclusions guarantee us that $$f_{15}(24)=f_{15}(26)=f_{15}(46)=-1,$$
			so $f_{15}$ is of type $2$. We conclude that $\{2,4,6\}\subset M_2$ for $f_{15}$. Suppose that for all $i,j\in \{9,10,11,\dots, n\}$ we have $f_{ij}\equiv 0$. Consequently, $f_{15}(ij)$ is the same value for all $i,j\in \{9,10,\dots, n\}$. However, for $n\geq 12$ it contradicts to the type of $f_{15}$. The last opportunity is that there is one more partial difference of the same structure as $f_{12}$ and $f_{56}$. Without loss of generality it is $f_{9,10}$ with defining elements $11$ and $12$. Since $f_{15}$ is of type $2$ one may find such a pair  of elements $(i,j)$ from $\{9,10,11,12\}$, that $f_{15}(ij)\neq 0$, but it contradicts to the structure of $f_{9,10}$.

		\end{enumerate}	
		\item $\overline{123}=\overline{124}=\overline{134}=\overline{234}=0$.\\
		In other words,  $123, \,124, \, 134,\, 234 \in X_2$. This is the last possible case for vertices $123$, $124$, $134$ and $234$. Hence, we may consider that similar equalities take place for all partial differences of type $1$ of the function $f$. Further arguments depend on values $\overline{12i}, \overline{34i}$ for $i\in [n]\setminus \{1,2,3,4\}$.
		\begin{enumerate}
			\item $\overline{12i}=\overline{34i}=1$ for $i\in [n]\setminus \{1,2,3,4\}$.
			In other words, all these vertices belong to $X_1$. Evidently, the vertex $123\in X_2$ and has $n-1$ neighbours from $X_2$  and $2n-8$ ones from $X_1$. Therefore, the quotient matrix of this partition equals	$$\left(
			\begin{array}{cc}
			3n-15 & 6 \\
			2n-8 & n-1 \\
			\end{array}
			\right)$$
			For the moment, we know that every vertex containing the pair $13$ or $24$ belongs to $X_2$. Let us analyse the partial difference $f_{15}$. We know that $f_{15}(23)=f_{15}(34)=1$. Consequently, $f_{15}$ is of type $1$ and one its defining element is $3$. Since $f_{15}(24)=0$, the second defining element belongs to $\{6,7,\dots, n\}$. Without loss of generality it will be $6$. By structure of the partial difference we obtain that $i56\in X_2$, for $i\in [n]\setminus \{1,3\}$. Since $f_{15}$ is of type $1$ and by arguments provided above we conclude that $156$, $356\in X_2$ too. Therefore, every vertex containing the pair $56$ belongs to $X_2$. Further we consider $f_{17}$ and by same arguments conclude that it is of type $1$ with defining elements $3$ and $8$. This process ends with $f_{1(n-1)}$. Finally, for $j\in \{1,2\dots \frac{n}{2}\}$ and $i\in [n]\setminus \{2j-1,\, 2j\}$, we have that $\{2j-1,2j,i\} \in X_2$. It is easy to see that all these vertices have $n-1$ neighbours from $X_2$, so all their remaining neighbours belong to $X_1$. Evidently, this partition is an instance of $\Pi_2$. 
			\item $\overline{12i}=0$ for some $i$. As usual, it is convenient to take $i=5$ and conclude that $125\in X_2$. In this case we have $f_{13}(25)=1$ and $f_{13}(24)=0$. Let $f_{13}$ be of type $1$. If $2$ is a defining element then the only possible second element is $4$, so $f_{13}(2i)=1$ and $12i\in X_2$ for $i\in \{5,6,\dots, n\}$. By simple counting we have that $123$ has at least $2n-5$ neighbours from $X_2$, and not more that $n-4$ ones from $X_1$. Hence, $p_{21}\leq n-4$ and we get a contradiction. Another possible case is when one defining element is $5$. If the second one is $4$ then we obtain $f_{13}(24)=-1$ but is must be $0$. If the second element $j$ is from the set $\{6,7,\dots, n\}$ then we have that $f_{13}(5k)=1$ for $k\in [n]\setminus \{1,3,5,j\}$, but $f_{13}(4k)\leq 0$.
			
			Consequently, we conclude that $f_{13}$ is of type $2$. Since $f_{13}(25)=1$ and $f_{13}(24)=0$, elements $2$ and $5$ belong to one defining set and $4$ to another one. Without loss of generality let us take $M_1=\{2,5,7,\dots, n-1\}$ and $M_2=\{4,6,8,\dots,n\}$. At the moment the following information on the structure of our partition is known 
			$$13i\in X_2 \text{ for } i\in [n]\setminus \{1,3\},   24i \in X_2 \text{ for } i\in [n]\setminus \{2,4\},$$
			$$14i,\, 23i \in X_1 \text{ for } i\in [n]\setminus \{1,2,3,4\},$$
			$$1ij \in X_2 \text{ for } i,j\in \{2,5,7,\dots, n-1\}, 3ij \in X_2 \text{ for } i,j\in \{4,6,\dots, n\},$$
			$$2ij \in X_2 \text{ for } i,j\in \{1,5,7,\dots, n-1\}, 4ij \in X_2 \text{ for } i,j\in \{3,6,8,\dots, n\},$$
			$$3ij \in X_1 \text{ for } i,j\in \{2,5,7,\dots, n-1\}, 1ij \in X_1 \text{ for } i,j\in \{4,6,\dots, n\},$$
			$$4ij \in X_1 \text{ for } i,j\in \{1,5,7,\dots, n-1\}, 2ij \in X_1 \text{ for } i,j\in \{3,6,8,\dots, n\}.$$
			Based on this structure it is easy to count that the vertex $123$ has $\frac{3}{2}n-3$ neighbours from $X_2$ and $\frac{3}{2}n-6$ ones from $X_1$. Consequently, the partition has the following quotient matrix
			$$\left(
			\begin{array}{cc}
			\frac{5}{2}n-15 & \frac{n}{2}+4 \\
			\frac{3}{2}n-6 & \frac{3}{2}n-3 \\
			\end{array}
			\right)$$
			Suppose that $579\in X_1$. Then $f_{15}(79)=f_{15}(23)=f_{15}(34)=1$ and $f_{15}$ is of type $2$. It follows from these equalities that $2$, $3$, $4$, $7$ and $9$ belong to one defining set. Hence, $f_{15}(47)=1$ and $147\in X_2$, but it is false. Since these arguments work not only for $579$ but for all $i,j,k\in M_1\setminus \{1,2\}$ and by previous equalities we conclude that 
			$$ijk \in X_2 \text{ for } i,j,k\in \{1,2,5,7,\dots, n-1\}.$$
			Similar arguments for the triple $\{6,8,10\}$ and the partial difference $f_{36}$ give us
			$$ijk \in X_2 \text{ for } i,j,k\in \{2,3,4,6,\dots, n\}.$$
			Simple calculations show that vertices $123$ and $134$ both have $\frac{3}{2}n-3$ neighbours from $X_2$. Consequently, $$12i \in X_1 \text{ for } i\in \{6,8, \dots, n\},$$ $$34i \in X_1 \text{ for } i\in \{7,9,\dots, n-1\}.$$
			As we established before every triple containing the pair $(1,3)$ or $(2,4)$ belongs to $X_2$.
			Our next goal is to show that there exists a perfect matching between $M_1$ and $M_2$ such that these pairs have same property.
			Consider the vertex $126\in X_1$. By simple counting we have that $126$ has $\frac{n}{2}+2$ neighbours from $X_2$. In order to build the partition we need to find two more neighbours. Without loss of generality let use take $156\in X_2$, by the structure of $f_{12}$ we immediately have that $256\in X_2$ too. Since $f_{13}(56)=0$ and by the structure of $f_{12}$ we conclude that $356,456 \in X_2$.     
			All these arguments allow us to claim that $f_{15}(23)=f_{15}(34)=1$ and $f_{15}(26)=-1$. Obviously, $f_{15}$ has type $1$ and with defining elements $3$ and $6$. Hence, $56i\in X_2$ for $i \in [n]\setminus \{5,6\}$. As a result of consideration of $126$ we find one more pair $(5,6)$ with same properties as for $(1,3)$ and $(2,4)$. In particular, $$15i, 25i \in X_1 \text{ for } i\in \{6,8, \dots, n\},$$ $$36i, 46i \in X_1 \text{ for } i\in \{7,9,\dots, n-1\}.$$ 
			Now let us consider the vertex $128\in X_1$. Again we need to find two more adjacent vertices from $X_2$. As we have just showed $158,258\in X_1$, so without loss of generality we have $178,278\in X_2$. By similar arguments we can show that $78$ is one more pair with required properties. 
			After considering all remaining elements we will show that all triples containing pairs $(1,3)$ $(2,4)$, $(5,6)$, $(7,8)$, $\dots$, $(n-1,n)$ belong to $X_2$. Besides that, all triples from $\{1,2,5,7,\dots, n-1\}$ and $\{3,4,6,8,\dots, n\}$ also belong to $X_2$. Direct calculations show that each of these vertices has $\frac{3}{2}n-3$ neighbours from $X_2$. Therefore, all their remaining neighbours are from $X_1$. Finally, we constructed a partition which is an instance of $\Pi_3$.

		\end{enumerate}
	\end{enumerate}

\end{proof}

\section{Local structure of equitable 2-partitions}\label{sec:Jn3t2}

Gavrilyuk and Goryainov \cite{GG_2013} use an algebraic tool to analyse the structure of an equitable 2-partition in $J(n,3)$. In particular, they classify all equitable 2-partitions in $J(n,3)$ when $n$ is odd, and all equitable 2-partitions in $J(n,3)$ for which both cells have the same size (such a partition corresponds to a symmetric quotient matrix). In this section, we introduce the tools used by Gavrilyuk and Goryainov \cite{GG_2013}, and suggest an alternative approach to analyse the remaining cases of $\lambda_2$-equitable 2-partitions of $J(n,3)$. 


We will work with equitable $2$-partitions of Johnson graphs in notations from Section \ref{sec:PartialDiff}. In particular, $X = \{X_1,X_2\}$ is a $\theta$-equitable $2$-partition of $J(n,3)$ for a non-principal eigenvalue $\theta$ of $\Gamma=J(n,3)$ with quotient matrix 
$$
A(\Gamma / X) =
\left(
\begin{array}{cc}
p_{11} & p_{12} \\
p_{21} & p_{22} \\
\end{array}
\right),
$$
where $p_{12}<p_{21}$ and  $p_{12}+p_{21}=2n-2$. Further, we can define indicator functions for vertices and subsets of vertices as in Section \ref{sec:PartialDiff}.
	
We have the following identities, relating the indicator of a vertex $abc$ of $\Gamma$ and the number of vertices in each row of $\Gamma(abc)$ that are in $X_1$.
\begin{lemma}\label{lem:localIndicator}
	For a vertex $abc$ of $J(n,3)$, the following equality holds
	\begin{equation}\label{eq:locIndEq}
	\overline{abc}\cdot(\theta+3) + p_{21} = \overline{ab\ast} + \overline{ac\ast} + \overline{bc\ast}.
	\end{equation}
\end{lemma}
\begin{proof}
	This can be shown by using a simple counting argument and the relationship between the quotient matrix and the number $\overline{\Gamma(abc)}$. For the full details, see \cite[Equation (2)]{GG_2013}.
\end{proof}

\begin{lemma}\label{lem:lemeq1}
	For any four distinct elements $a,b,c,d \in  \left[n\right]$, the following condition holds:
	\begin{equation}\label{eq:eq1}
	\overline{ab\ast} - \overline{cd\ast} = \frac{\theta+3}{2}(\overline{abc} + \overline{abd} - \overline{acd} - \overline{bcd}).
	\end{equation}
\end{lemma}
\begin{proof}
	This result is a rearrangement and application of \cite[Lemma 3.5]{GG_2013}. By Lemma \ref{lem:localIndicator}, we have the following equalities:
	$$
	\overline{abc}\cdot(\theta+3) + p_{21} = \overline{ab\ast} + \overline{ac\ast} + \overline{bc\ast},
	$$
	$$
	\overline{abd}\cdot(\theta+3) + p_{21} = \overline{ab\ast} + \overline{ad\ast} + \overline{bd\ast},
	$$
	$$
	-\overline{acd}\cdot(\theta+3) - p_{21} = -\overline{ac\ast} - \overline{ad\ast} - \overline{cd\ast},
	$$
	$$
	-\overline{bcd}\cdot(\theta+3) - p_{21} = -\overline{bc\ast} - \overline{bd\ast} - \overline{cd\ast}.
	$$
	Then sum up these four equalities and divide by $2$ to see the result.
\end{proof}

Lemma \ref{lem:lemeq1} shows that for any two vertex-disjoint maximum cliques $ab\ast$ and $cd\ast$ in $J(n,3)$, the difference $\overline{ab\ast} - \overline{cd\ast}$ is determined by the eigenvalue $\theta$ and the four values		$\overline{abc}$, $\overline{abd}$, $\overline{acd}$ and $\overline{bcd}$.

\begin{lemma}\label{lem:eq2}
	For any five distinct elements $a,b,c,d,e \in \left[n\right]$, the following condition holds:
	\begin{equation}\label{eq:eq2}
	\overline{ab\ast} - \overline{ac\ast} = \frac{\theta+3}{2}(\overline{abd} + \overline{abe} + \overline{cde} - \overline{acd} - \overline{ace} - \overline{bde}).
	\end{equation}
\end{lemma}
\begin{proof}
	By Lemma \ref{lem:lemeq1}, we have the following equalities:
	$$
	\overline{ab\ast} - \overline{cd\ast} = \frac{\theta+3}{2}(\overline{abc} + \overline{abd} - \overline{acd} - \overline{bcd}),
	$$
	$$
	\overline{cd\ast} - \overline{be\ast} = \frac{\theta+3}{2}(\overline{bcd} + \overline{cde} - \overline{bce} - \overline{bde}),
	$$
	$$
	\overline{be\ast} - \overline{ac\ast} = \frac{\theta+3}{2}(\overline{abe} + \overline{bce} - \overline{abc} - \overline{ace}).
	$$
	Then we sum up these three equalities to see the result.
\end{proof}

\subsection{Local structure of $\lambda_2$-equitable $2$-partitions}\label{sec:Jt2}

From now on, we will assume the partition $X=\{X_1,X_2\}$ is $\lambda_2$-equitable. Without loss of generality, we let $p_{11}\geq p_{22}$. By Lemma \ref{lem:theta}, we have 
$$
4(n-4) = k+\lambda_2 = p_{11}+p_{22},
$$
so we must have $p_{11}\geq 2n-8$. If $p_{11}=2n-8$, we see that $p_{11}=p_{22}$ and $p_{12}=p_{21}=n-1$. All $\lambda_2$-equitable $2$-partitions with such a quotient matrix are found in \cite{GG_2013}. For the rest of the paper we assume $p_{11}\geq 2n-7$. 

Let $abc$ be a vertex of $\Gamma$, such that $abc\in X_1$. By Equation \eqref{eq:eq2}, we have 

\begin{equation}\label{eq:rowdiff}
\overline{ab\ast} - \overline{ac\ast} = \frac{n-4}{2}(\overline{abd} + \overline{abe} + \overline{cde} - \overline{acd} - \overline{ace} - \overline{bde}).
\end{equation}

Using this and the fact that $0\leq \overline{ij\ast}\leq n-2$, we see that the difference  $\overline{ab\ast} - \overline{ac\ast}$ is equal to $h(n-4)/2$, where $h\in \{-2,-1,0,1,2\}$.  

Now we fix a vertex $abc\in V(\Gamma)$. Without loss of generality, we may assume that $\overline{ab\ast}\geq\overline{ac\ast}\geq\overline{bc\ast}$. Now we define the \emph{difference tuple} of $abc$ as the tuple $(\overline{ab\ast} - \overline{ac\ast},\overline{ac\ast} - \overline{bc\ast})$. We have the following possibilities for difference tuples:
~
\begin{varwidth}{\textwidth}
	\begin{tasks}[label=(\Roman*),label-width={0.7cm},column-sep={2cm}](2)\label[cases]{tasks:I-VI}
		\task\label{case:I}  $\left(n-4,0 \right)$
		\task\label{case:II} $\left( (n-4)/2, (n-4)/2 \right)$
		\task\label{case:III} $\left( 0, n-4\right)$
		\task\label{case:IV} $\left( (n-4)/2,0 \right)$
		\task\label{case:V} $\left( 0, (n-4)/2 \right)$
		\task\label{case:VI} $\left( 0,0 \right)$
	\end{tasks}
\end{varwidth}

\subsection{Case analysis}

To understand the arguments in this section, we will be using arrays of the following form. Consider the vertex $abc$ of $\Gamma$ and $X=\{X_1,X_2\}$, a $\lambda_2$-equitable 2-partition of $\Gamma$. The \emph{nb-array} of $abc$ is a $3\times(2m-3)$-array corresponding to the neighbourhood $\Gamma(abc)$, and at each vertex $u\in \Gamma(abc)$ the corresponding entry of the array is equal to $\overline{u}$. The rows are given by the comments on the right. If needed, the indices of columns are given below the braces.

$$
\Gamma(abc):
\underbrace{
	\begin{array}{c}
	\overline{abd}\\
	\overline{acd}\\
	\overline{bcd}
	\end{array}
}_{d}
\underbrace{
	\begin{array}{cc}
	\overline{abe} & \overline{abf}\\
	\overline{ace} & \overline{acf}\\
	\overline{bce} & \overline{bcf}
	\end{array}
}_{\{e,f\}}
\begin{array}{ccccccccc}
\dots & \dots & \dots & \dots & \dots & \dots & \dots & \dots & \dots\\
\dots & \dots & \dots & \dots & \dots & \dots & \dots & \dots & \dots\\
\dots & \dots & \dots & \dots & \dots & \dots & \dots & \dots & \dots
\end{array}
\begin{array}{c}  
~\leftarrow ab\text{-row} \\
~\leftarrow ac\text{-row} \\
~\leftarrow bc\text{-row} 
\end{array}
$$
When considering the cases \labelcref{case:I,case:II,case:III,case:IV,case:V,case:VI}, we can use Equation \eqref{eq:rowdiff} to deduce restrictions on the corresponding nb-arrays. Note that we have assumed that\linebreak $\overline{ab\ast}\geq\overline{ac\ast}\geq\overline{bc\ast}$.

\begin{lemma}\label{lem:eqq11}
	Let $abc$ be a vertex of $\Gamma$ and $\overline{abc}=1$ Then:
	\begin{enumerate}
		\item If $\overline{ab\ast}=\overline{ac\ast}$, there are no indices $d,e\in \left[n\right]\setminus \{a,b,c\}$ such that $\overline{abd}=\overline{abe}=1$ and $\overline{acd}=\overline{ace}=0$.
		\item If  $\overline{ab\ast}>\overline{ac\ast}$, there are no indices $d,e\in \left[n\right]\setminus \{a,b,c\}$ such that 
		\begin{enumerate}
			\item $\overline{acd}=1$ and $\overline{abe}=\overline{abd}=\overline{ace}=0$, 
			\item $\overline{abd}=0$ and $\overline{abe}=\overline{acd}=\overline{ace}=1$, or	
			\item $\overline{acd}=\overline{ace}=1$ and $\overline{abd}=\overline{abe}=0$.
		\end{enumerate}
	\end{enumerate}
\end{lemma}
\begin{proof}
	1. Suppose $d,e\in \left[n\right]\setminus \{a,b,c\}$ such that $\overline{abd}=\overline{abe}=1$ and $\overline{acd}=\overline{ace}=0$. Then by Equation \eqref{eq:rowdiff},
	$$
	0=2+\overline{cde}-\overline{bde}>0,
	$$
	giving a contradiction.
	
	2.(a) Suppose $d,e\in \left[n\right]\setminus \{a,b,c\}$ such that $\overline{acd}=1$ and $\overline{abe}=\overline{abd}=\overline{ace}=0$. By Equation \eqref{eq:rowdiff},
	$$
	0<\overline{cde}-1-\overline{bde}<1,
	$$
	giving a contradiction. The result of 2.(b) and (c) follows similarly.
\end{proof}

Now we use Lemma \ref{lem:eqq11} to enumerate all possible nb-arrays in each case. In fact, Cases (I) and (II) are shown to be impossible for all $n>12$. In Cases (III) and (IV) we can characterise the quotient matrix in terms of $n$.

\begin{lemma}[Case \labelcref{case:I}]
	Let $abc$ be a vertex of $\Gamma$ with $\overline{abc}=1$ and difference tuple $(n-4,0)$. Then $n\leq 6$
\end{lemma}
\begin{proof}
	In this case, we observe that $p_{11}=\overline{\Gamma(abc)}$ is at most $(n-3)+1+1=n-1$. Therefore $n-1\geq 2n-7$ by assumption, and so $n\leq 6$.
\end{proof}

\begin{lemma}[Case \labelcref{case:II}]
	Let $abc$ be a vertex of $\Gamma$ with $\overline{abc}=1$ and difference tuple $((n-4)/2,(n-4)/2)$. Then $n\leq 8$
\end{lemma}
\begin{proof}
	In this case, we observe that $p_{11}=\overline{\Gamma(abc)}$ is at most $(n-3)+(n-2)/2+1=3(n-2)/2$. Therefore $3(n-2)/2\geq 2n-7$ by assumption, and so $n\leq 8$.
\end{proof}

\begin{lemma}[Case \labelcref{case:III}]
	Let $abc$ be a vertex of $\Gamma$ with $\overline{abc}=1$ and difference tuple $(0,n-4)$. Then $abc$ has the following nb-array (up to reordering).
	$$
	\Gamma(abc):
	\begin{array}{c}
	1\\
	1\\
	1
	\end{array}
	\begin{array}{ccccccccccc}
	1 & 1 & \dots & \dots & \dots & \dots & \dots & 1 & 1 & 1 & 1\\
	1 & 1 & \dots & \dots & \dots & \dots & \dots & 1 & 1 & 1 & 1\\
	0 & 0 & \dots & \dots & \dots & \dots & \dots & 0 & 0 & 0 & 0
	\end{array}
	\begin{array}{c}  
	~\leftarrow ab\text{-row} \\
	~\leftarrow ac\text{-row} \\
	~\leftarrow bc\text{-row} 
	\end{array}
	$$
\end{lemma}
\begin{proof}	
	In this case, we observe that $p_{11}=\overline{\Gamma(abc)}$ is at most $(n-3)+n-3+1=2n-5$. Furthermore, we must have 
	$\overline{ab\ast}\geq n-2$, as otherwise $p_{11}\leq (n-4)+(n-4)+0=2n-8$. Therefore, $\overline{ab\ast}=n-2$ and $p_{11}=2n-5$. 
	
	From this we deduce that the only possible nb-arrays (up to reordering) are the following.
\end{proof}

\begin{lemma}[Case \labelcref{case:IV}]
	Let $abc$ be a vertex of $\Gamma$ with $\overline{abc}=1$ and difference tuple $((n-4)/2,0)$. Then $abc$ has one of the following nb-arrays (up to reordering).
	\begin{enumerate}[(i)]
		\item\label{case:IVi} 
		$$
		\Gamma(abc):
		\begin{array}{ccccc}
		1 & \dots & \dots & 1 & 1\\
		1 & \dots & \dots & 1 & 1\\
		1 & \dots & \dots & 1 & 1
		\end{array}
		\begin{array}{ccccc}
		1 & 1 & \dots & \dots & 1\\
		0 & 0 & \dots & \dots & 0\\
		0 & 0 & \dots & \dots & 0
		\end{array}
		\begin{array}{c}  
		~\leftarrow ab\text{-row} \\
		~\leftarrow ac\text{-row} \\
		~\leftarrow bc\text{-row} 
		\end{array}
		$$
		\item\label{case:IVii}
		$$
		\Gamma(abc):
		\begin{array}{c}
		1\\
		1\\
		0
		\end{array}
		\begin{array}{c}
		1\\
		0\\
		1
		\end{array}
		\begin{array}{cccc}
		1 & \dots & \dots & 1\\
		1 & \dots & \dots & 1\\
		1 & \dots & \dots & 1
		\end{array}
		\begin{array}{cccc}
		1 & \dots & \dots & 1\\
		0 & \dots & \dots & 0\\
		0 & \dots & \dots & 0
		\end{array}
		\begin{array}{c}  
		~\leftarrow ab\text{-row} \\
		~\leftarrow ac\text{-row} \\
		~\leftarrow bc\text{-row} 
		\end{array}
		$$
	\end{enumerate}
\end{lemma}
\begin{proof}
	In this case, we observe that $p_{11}=\overline{\Gamma(abc)}$ is at most $(n-3)+(n-2)/2+(n-2)/2=2n-5$. Furthermore, we must have  $\overline{ab\ast}\geq n-2$, as otherwise  $p_{11}\leq (n-3-1) + 2(n-3-1-(n-4)/2)= 2n-8$. Therefore, $\overline{ab\ast}=n-2$ and $p_{11}=2n-5$. 
	
	Using Lemma \ref{lem:eqq11} 1., we see that any two distinct columns of the nb-array of $\Gamma(abc)$ (up to reordering) cannot look like one of the following pairs. 
	$$
	\begin{array}{cc}
	1 & 1\\
	0 & 0\\
	1 & 1
	\end{array} \text{ or } 
	\begin{array}{cc}
	1 & 1\\
	1 & 1\\
	0 & 0
	\end{array}
	\begin{array}{c}  
	~\leftarrow ab\text{-row} \\
	~\leftarrow ac\text{-row} \\
	~\leftarrow bc\text{-row} 
	\end{array}
	$$
	Furthermore, if we have a single column from one of the above pairs present in the nb-array, then a column of the other pair must also be present, as $\overline{ac\ast}=\overline{bc\ast}$.
	
	Therefore, the only possible nb-arrays (up to reordering) are the above.
\end{proof}

\begin{lemma}[Case \labelcref{case:V}]
	Let $abc$ be a vertex of $\Gamma$ with $\overline{abc}=1$ and difference tuple $(0,(n-4)/2)$. Then $abc$ has one of the following nb-arrays (up to reordering).
	\begin{enumerate}[(i)]
		\item\label{case:Vi} 
		$$
		\Gamma(abc):
		\begin{array}{cccc}
		1 & \dots & \dots & 1\\
		1 & \dots & \dots & 1\\
		1 & \dots & \dots & 1
		\end{array}
		\begin{array}{cccc}
		1 & \dots & \dots & 1\\
		1 & \dots & \dots & 1\\
		0 & \dots & \dots & 0
		\end{array}
		\begin{array}{cccc}
		0 & \dots & \dots & 0\\
		0 & \dots & \dots & 0\\
		0 & \dots & \dots & 0
		\end{array}
		\begin{array}{c}  
		~\leftarrow ab\text{-row} \\
		~\leftarrow ac\text{-row} \\
		~\leftarrow bc\text{-row} 
		\end{array}
		$$
		(This case occurs in the partition $\Pi_1$ found in Lemma \ref{thm:knJn3}, for which $p_{11}=5n/2-7$.) 
		\item\label{case:Vii}
		$$
		\Gamma(abc):
		\begin{array}{c}
		1\\
		0\\
		0
		\end{array}
		\begin{array}{c}
		0\\
		1\\
		0
		\end{array}
		\begin{array}{cccc}
		1 & \dots & \dots & 1\\
		1 & \dots & \dots & 1\\
		1 & \dots & \dots & 1
		\end{array}
		\begin{array}{cccc}
		1 & \dots & \dots & 1\\
		1 & \dots & \dots & 1\\
		0 & \dots & \dots & 0
		\end{array}
		\begin{array}{cccc}
		0 & \dots & \dots & 0\\
		0 & \dots & \dots & 0\\
		0 & \dots & \dots & 0
		\end{array}
		\begin{array}{c}  
		~\leftarrow ab\text{-row} \\
		~\leftarrow ac\text{-row} \\
		~\leftarrow bc\text{-row} 
		\end{array}
		$$
		(This case occurs in the partition $\Pi_3$ found in Lemma \ref{thm:knJn3}, for which $p_{11}=5n/2-13$.) 
	\end{enumerate}
\end{lemma}
\begin{proof}	
	Using Lemma \ref{lem:eqq11}, we see that any two distinct columns of the nb-array of $\Gamma(abc)$ (up to reordering) cannot look like the following pairs.
	$$
	\underbrace{\begin{array}{cc}
		1 & 1\\
		0 & 0\\
		* & *
		\end{array} \text{ or } 
		\begin{array}{cc}
		0 & 0\\
		1 & 1\\
		* & *
		\end{array}}_{\text{ by Lemma \ref{lem:eqq11} 1.}}\qquad
	\underbrace{\begin{array}{cc}
		0 & 0\\
		* & *\\
		1 & 0
		\end{array} \text{ or } 
		\begin{array}{cc}
		* & *\\
		0 & 0\\
		1 & 0
		\end{array}}_{\text{ by Lemma \ref{lem:eqq11} 2(a).} }
	$$
	$$
	\underbrace{\begin{array}{cc}
		0 & 1\\
		* & *\\
		1 & 1
		\end{array} \text{ or } 
		\begin{array}{cc}
		* & *\\
		0 & 1\\
		1 & 1
		\end{array}}_{\text{ by Lemma \ref{lem:eqq11} 2(b).} }\qquad
	\underbrace{\begin{array}{cc}
		* & *\\
		0 & 0\\
		1 & 1
		\end{array} \text{ or }
		\begin{array}{cc}
		0 & 0\\
		* & *\\
		1 & 1
		\end{array}}_{\text{ by Lemma \ref{lem:eqq11} 2(c).} }
	$$
	(where any entry $\ast$ can take either of the values 0 or 1).
	
	Now suppose we have a column 
	$$
	\begin{array}{c}
	* \\
	0 \\
	1
	\end{array}\begin{array}{c}  
	~\leftarrow ab\text{-row} \\
	~\leftarrow ac\text{-row} \\
	~\leftarrow bc\text{-row} 
	\end{array}
	$$
	By the above restrictions on pairs of columns, any other column of the 
	nb-array must look like
	$$
	\begin{array}{c}
	* \\
	1 \\
	0
	\end{array}\begin{array}{c}  
	~\leftarrow ab\text{-row} \\
	~\leftarrow ac\text{-row} \\
	~\leftarrow bc\text{-row} 
	\end{array}
	$$
	This shows us that $\overline{ac\ast}=n-3$ and $\overline{bc\ast}=2$. But this gives a contradiction to the fact that $\overline{ac\ast}-\overline{bc\ast}=(n-4)/2$ and $n>6$.
	
	By a similar argument, we can show that we cannot have the column 
	$$
	\begin{array}{c}
	0 \\
	* \\
	1
	\end{array}\begin{array}{c}  
	~\leftarrow ab\text{-row} \\
	~\leftarrow ac\text{-row} \\
	~\leftarrow bc\text{-row} 
	\end{array},
	$$
	and so the possible columns of the nb-array  are 
	$$
	\begin{array}{c}
	1 \\
	1 \\
	1
	\end{array}\qquad\begin{array}{c}
	1 \\
	1 \\
	0
	\end{array}\qquad\begin{array}{c}
	0 \\
	0 \\
	0
	\end{array}\qquad\begin{array}{c}
	1 \\
	0 \\
	0
	\end{array}\qquad
	\begin{array}{c}
	0 \\
	1 \\
	0
	\end{array}\begin{array}{c}  
	~\leftarrow ab\text{-row} \\
	~\leftarrow ac\text{-row} \\
	~\leftarrow bc\text{-row} 
	\end{array},
	$$
	By the restrictions on pairs of columns from Lemma \ref{lem:eqq11} 1., the last two columns in the list can occur at most once. If one of the last two columns are present in the nb-array, the other must be present, as $\overline{ab\ast}=\overline{ac\ast}$ (this will be called case (\labelcref{case:Vii})).
	
	From the discussion above, we deduce that the only possible nb-arrays (up to reordering) are the above.
\end{proof}

\begin{lemma}[Case \labelcref{case:VI}]
	Let $abc$ be a vertex of $\Gamma$ with $\overline{abc}=1$ and difference tuple $(0,0)$. Then $abc$ has one of the following nb-arrays (up to reordering).
	\begin{enumerate}[(i)]
		\item\label{case:VIi} 
		$$
		\Gamma(abc):
		\begin{array}{ccccccc}
		1 & 1 & 1 & 1 & \dots & \dots & 1\\
		1 & 1 & 1 & 1 & \dots & \dots & 1\\
		1 & 1 & 1 & 1 & \dots & \dots & 1
		\end{array}
		\begin{array}{cccc}
		0 & \dots & \dots & 0\\
		0 & \dots & \dots & 0\\
		0 & \dots & \dots & 0
		\end{array}
		\begin{array}{c}  
		~\leftarrow ab\text{-row} \\
		~\leftarrow ac\text{-row} \\
		~\leftarrow bc\text{-row} 
		\end{array}
		$$
		\item\label{case:VIii}
		$$
		\Gamma(abc):
		\begin{array}{c}
		1\\
		1\\
		0
		\end{array}
		\begin{array}{c}
		0\\
		0\\
		1
		\end{array}
		\begin{array}{ccccc}
		1 & 1 & \dots & \dots & 1\\
		1 & 1 & \dots & \dots & 1\\
		1 & 1 & \dots & \dots & 1
		\end{array}
		\begin{array}{cccc}
		0 & \dots & \dots & 0\\
		0 & \dots & \dots & 0\\
		0 & \dots & \dots & 0
		\end{array}
		\begin{array}{c}  
		~\leftarrow ab\text{-row} \\
		~\leftarrow ac\text{-row} \\
		~\leftarrow bc\text{-row} 
		\end{array}
		$$
		\item\label{case:VIiii}
		$$
		\Gamma(abc):
		\begin{array}{c}
		1\\
		0\\
		0
		\end{array}
		\begin{array}{c}
		0\\
		1\\
		0
		\end{array}
		\begin{array}{c}
		0\\
		0\\
		1
		\end{array}
		\begin{array}{cccc}
		1 & \dots & \dots & 1\\
		1 & \dots & \dots & 1\\
		1 & \dots & \dots & 1
		\end{array}
		\begin{array}{cccc}
		0 & \dots & \dots & 0\\
		0 & \dots & \dots & 0\\
		0 & \dots & \dots & 0
		\end{array}
		\begin{array}{c}  
		~\leftarrow ab\text{-row} \\
		~\leftarrow ac\text{-row} \\
		~\leftarrow bc\text{-row} 
		\end{array}
		$$
		(This case occurs in the partition $\Pi_2$ found in Lemma \ref{thm:knJn3}, for which $p_{11}=3(n-5)$.) 
		
		\item\label{case:VIiv}
		$$
		\Gamma(abc):
		\begin{array}{c}
		1\\
		1\\
		0
		\end{array}
		\begin{array}{c}
		1\\
		0\\
		1
		\end{array}
		\begin{array}{c}
		0\\
		1\\
		1
		\end{array}
		\begin{array}{cccc}
		1 & \dots & \dots & 1\\
		1 & \dots & \dots & 1\\
		1 & \dots & \dots & 1
		\end{array}
		\begin{array}{cccc}
		0 & \dots & \dots & 0\\
		0 & \dots & \dots & 0\\
		0 & \dots & \dots & 0
		\end{array}
		\begin{array}{c}  
		~\leftarrow ab\text{-row} \\
		~\leftarrow ac\text{-row} \\
		~\leftarrow bc\text{-row} 
		\end{array}
		$$
	\end{enumerate}
\end{lemma}
\begin{proof}
	Using Lemma \ref{lem:eqq11} 1., we see that any two distinct columns of the nb-array of $\Gamma(abc)$ (up to reordering) cannot look like one of the following pairs. 
	$$
	\begin{array}{cc}
	1 & 1\\
	0 & 0\\
	* & *
	\end{array} \qquad 
	\begin{array}{cc}
	0 & 0\\
	1 & 1\\
	* & *
	\end{array} \qquad
	\begin{array}{cc}
	1 & 1\\
	* & *\\
	0 & 0
	\end{array} \qquad 
	\begin{array}{cc}
	0 & 0\\
	* & *\\
	1 & 1
	\end{array} \qquad
	\begin{array}{cc}
	* & *\\
	1 & 1\\
	0 & 0
	\end{array} \qquad 
	\begin{array}{cc}
	* & *\\
	0 & 0\\
	1 & 1
	\end{array}$$
	(where any entry $\ast$ can take either of the values 0 or 1).

	Suppose we have a column in the nb-array with exactly two entries equal to 1 (this will split into the cases (\labelcref{case:VIii}) and (\labelcref{case:VIiv}) in the following). Without loss of generality, let this column be as follows.
	$$
	\begin{array}{c}
	1 \\
	1 \\
	0
	\end{array}\begin{array}{c}  
	~\leftarrow ab\text{-row} \\
	~\leftarrow ac\text{-row} \\
	~\leftarrow bc\text{-row} 
	\end{array}
	$$
	As $\overline{ac\ast}=\overline{bc\ast}$, at least one other column must be of the form
	$$
	\begin{array}{c}
	* \\
	0 \\
	1
	\end{array}\begin{array}{c}  
	~\leftarrow ab\text{-row} \\
	~\leftarrow ac\text{-row} \\
	~\leftarrow bc\text{-row} 
	\end{array}
	$$
	Suppose we are in the case (this will be called case (\labelcref{case:VIiv})) that we have columns 
	$$
	\begin{array}{cc}
	1 & 1 \\
	1 & 0 \\
	0 & 1
	\end{array}\begin{array}{c}  
	~\leftarrow ab\text{-row} \\
	~\leftarrow ac\text{-row} \\
	~\leftarrow bc\text{-row} 
	\end{array}
	$$
	We use the fact that $\overline{ab\ast}=\overline{ac\ast}$ and the restrictions on the columns to find that we must have the three columns
	$$
	\begin{array}{ccc}
	1 & 1 & 0\\
	1 & 0 & 1\\
	0 & 1 & 1
	\end{array}\begin{array}{c}  
	~\leftarrow ab\text{-row} \\
	~\leftarrow ac\text{-row} \\
	~\leftarrow bc\text{-row} 
	\end{array}.
	$$
	From here, it is straightforward to see that each of the remaining columns must all three entries equal.  
	
	Now suppose we are in the case (this will be called case (\labelcref{case:VIii})) where instead, we have columns
	$$
	\begin{array}{cc}
	1 & 0 \\
	1 & 0 \\
	0 & 1
	\end{array}\begin{array}{c}  
	~\leftarrow ab\text{-row} \\
	~\leftarrow ac\text{-row} \\
	~\leftarrow bc\text{-row} 
	\end{array}
	$$
	Using the restrictions on pairs of columns, we see that any other column cannot be of the form
	$$
	\begin{array}{c}
	* \\
	0 \\
	1
	\end{array} \text{ or }\begin{array}{c}
	0 \\
	* \\
	1
	\end{array}\begin{array}{c}  
	~\leftarrow ab\text{-row} \\
	~\leftarrow ac\text{-row} \\
	~\leftarrow bc\text{-row} 
	\end{array}
	$$
	This means that any non-zero entry which contributes positively to the sum $\overline{bc\ast}$ must be in a column with all entries equal to 1. As $\overline{ab\ast}=\overline{ac\ast}=\overline{bc\ast}$, we deduce that all columns with at least one entry equal to 1 must have all entries equal to 1. 
	
	Now suppose we have no columns with exactly two entries equal to 1 (this will split into the cases (\labelcref{case:VIi}) and (\labelcref{case:VIiii}) in the following). Further suppose there is a column with exactly one entry equal to 1 (this will be called case (\labelcref{case:VIiii})). Without loss of generality, let this column be as follows.
	$$
	\begin{array}{c}
	1 \\
	0 \\
	0
	\end{array}\begin{array}{c}  
	~\leftarrow ab\text{-row} \\
	~\leftarrow ac\text{-row} \\
	~\leftarrow bc\text{-row} 
	\end{array}
	$$
	By the restrictions on the pairs of columns, the assumption we have no columns with two entries equal to 1, and $\overline{ab\ast}=\overline{ac\ast}=\overline{bc\ast}$,
	we must have the three columns.
	$$
	\begin{array}{ccc}
	1 & 0 & 0\\
	0 & 1 & 0\\
	0 & 0 & 1
	\end{array}\begin{array}{c}  
	~\leftarrow ab\text{-row} \\
	~\leftarrow ac\text{-row} \\
	~\leftarrow bc\text{-row} 
	\end{array}
	$$
	From here, it is straightforward to see that each of the remaining columns must all three entries equal.  
	
	The last case is where all entries of a single column are equal (this will be called case (\labelcref{case:VIi})). 
	
	In the discussion above, we use Lemma \ref{lem:eqq11} 1. to deduce that the only possible nb-arrays (up to reordering) are the following.
\end{proof}

\section{Local structure for large $n$}\label{sec:largen}

In this section, we will see that for $n$ large enough, many of the above nb-arrays cannot occur. In particular, for $n>14$ we have only three possible nb-arrays.

\subsection{Removing Cases \labelcref{case:III} and \labelcref{case:IV}(\labelcref{case:IVii})}

From now on, we will assume we do not see the cases \labelcref{case:I} or \labelcref{case:II} for any vertex in $\Gamma$. Note that when $n>8$, we know that these cases cannot occur. We will then prove that the cases \labelcref{case:III} and \labelcref{case:IV}(\labelcref{case:IVii}) cannot occur. First we will prove that these cases always occur together.

\begin{lemma}\label{lem:cIIIcIVii}
	Let $abc$ be a vertex of $\Gamma$ with $\overline{abc}=1$, and let $d,e\in \left[n\right]\setminus \{a,b,c\}$ be distinct. Then we have nb-array
	$$
	\Gamma(abc):
	\underbrace{
		\begin{array}{c}
		1\\
		1\\
		1
		\end{array}
	}_{d}
	\begin{array}{ccccccccccc}
	1 & 1 & \dots & \dots & \dots & \dots & \dots & 1 & 1 & 1 & 1\\
	1 & 1 & \dots & \dots & \dots & \dots & \dots & 1 & 1 & 1 & 1\\
	0 & 0 & \dots & \dots & \dots & \dots & \dots & 0 & 0 & 0 & 0
	\end{array}
	\begin{array}{c}  
	~\leftarrow ab\text{-row} \\
	~\leftarrow ac\text{-row} \\
	~\leftarrow bc\text{-row} 
	\end{array}
	$$
	if and only if we have nb-array
	$$
	\Gamma(abe):
	\underbrace{
		\begin{array}{c}
		1\\
		1\\
		0
		\end{array}
	}_{c}
	\underbrace{
		\begin{array}{c}
		1\\
		0\\
		1
		\end{array}
	}_{d}
	\underbrace{
		\begin{array}{cccc}
		1 & \dots & \dots & 1\\
		1 & \dots & \dots & 1\\
		1 & \dots & \dots & 1
		\end{array}
	}_{}
	\underbrace{
		\begin{array}{cccc}
		1 & \dots & \dots & 1\\
		0 & \dots & \dots & 0\\
		0 & \dots & \dots & 0
		\end{array}
	}_{}
	\begin{array}{c}  
	~\leftarrow ab\text{-row} \\
	~\leftarrow ae\text{-row} \\
	~\leftarrow be\text{-row} 
	\end{array}
	$$
\end{lemma}
\begin{proof}
	($\implies$) Suppose we have the nb-array for $abc$. Then by applying Equation \eqref{eq:rowdiff} to the rows $ab\ast,bc\ast$, we have
	$$
	2=\overline{abd}+\overline{abe}+\overline{cde}-\overline{bcd}-\overline{bce}-\overline{ade}=1+\overline{cde}-\overline{ade}.
	$$
	Therefore, we must have $\overline{cde}=1,\overline{ade}=0$. Applying Equation \eqref{eq:rowdiff} to the rows $ab\ast,ac\ast$, we have $\overline{bde}=1$.

	Further application of Equation \eqref{eq:rowdiff} to the rows $ab\ast,ae\ast$, and using the above, we see that
	$$
	\overline{ab\ast}-\overline{ae\ast}=\frac{n-4}{2}(\overline{abc}+\overline{abe}+\overline{cde}-\overline{ace}-\overline{ade}-\overline{bcd})=\frac{n-4}{2},
	$$
	and so $\overline{ae\ast}=n/2$.
	
	Now consider the nb-array of $abe$. We know the values for columns with indices $c,d$. We have assumed case \labelcref{case:II} does not occur, so the nb-array of $abe$ must be in case \labelcref{case:IV}. Using our knowledge of the $c$-column and $d$-column, we see the nb-array of $abe$ is in case \labelcref{case:IV}(\labelcref{case:IVii}).
	
	($\impliedby$) Suppose the nb-array of $abe$ is of the form above.
	
	Applying Equation \eqref{eq:rowdiff} to the rows $ab\ast,ad\ast$ and $ac\ast,ab\ast$, we see that
	$$
	\overline{ab\ast}-\overline{ad\ast}=\frac{n-4}{2}(\overline{abc}+\overline{abe}+\overline{cde}-\overline{acd}-\overline{ade}-\overline{bce})=\frac{n-4}{2}(2+\overline{cde}-\overline{acd}),
	$$
	and
	$$
	\overline{ac\ast}-\overline{ab\ast}=\frac{n-4}{2}(\overline{acd}+\overline{ace}+\overline{bde}-\overline{abd}-\overline{abe}-\overline{cde})=\frac{n-4}{2}(\overline{acd}-\overline{cde}).
	$$
	Summing these two together, we see that
	$$\overline{ac\ast}-\overline{ad\ast}=n-4$$
	and thus $\overline{ac\ast}\geq n-4$.
	
	Consider the nb-array of vertex $abc$. As the difference $\overline{ab\ast}-\overline{ac\ast}$ is an integer multiple of $(n-4)/2$ and $\overline{ab\ast}=n-2$, we must have $\overline{ac\ast}=n-2$, and the nb-array of $abc$ must be in case \labelcref{case:III}.
\end{proof}

Now we will work to prove that case \labelcref{case:III} leads to a contradiction, proving that cases \labelcref{case:III} and \labelcref{case:IV}(\labelcref{case:IVii}) cannot occur (when $n>8$).

\begin{lemma}\label{lem:cIIIij}
	Let $abc$ be a vertex of $\Gamma$ with $\overline{abc}=1$, and let $d\in \left[n\right]\setminus \{a,b,c\}$ be such that we have nb-array
	$$
	\Gamma(abc):
	\underbrace{
		\begin{array}{c}
		1\\
		1\\
		1
		\end{array}
	}_{d}
	\begin{array}{ccccccccccc}
	1 & 1 & \dots & \dots & \dots & \dots & \dots & 1 & 1 & 1 & 1\\
	1 & 1 & \dots & \dots & \dots & \dots & \dots & 1 & 1 & 1 & 1\\
	0 & 0 & \dots & \dots & \dots & \dots & \dots & 0 & 0 & 0 & 0
	\end{array}
	\begin{array}{c}  
	~\leftarrow ab\text{-row} \\
	~\leftarrow ac\text{-row} \\
	~\leftarrow bc\text{-row} 
	\end{array}
	$$
	
	Then:
	\begin{enumerate}
		\item The nb-array of $abd$ is 
		$$\Gamma(abd):
		\underbrace{
			\begin{array}{c}
			1\\
			1\\
			1
			\end{array}
		}_{c}
		\begin{array}{ccccccccccc}
		1 & 1 & \dots & \dots & \dots & \dots & \dots & 1 & 1 & 1 & 1\\
		1 & 1 & \dots & \dots & \dots & \dots & \dots & 1 & 1 & 1 & 1\\
		0 & 0 & \dots & \dots & \dots & \dots & \dots & 0 & 0 & 0 & 0
		\end{array}
		\begin{array}{c}  
		~\leftarrow ab\text{-row} \\
		~\leftarrow bd\text{-row} \\
		~\leftarrow ad\text{-row} 
		\end{array}.
		$$
		\item For any $i,j\in \left[n\right]\setminus \{a,b,c,d\}$ distinct, we have $\overline{aij}=\overline{bij}=\overline{cij}$.
		\item For any $e\in \left[n\right]\setminus \{a,b,c,d\}$, there exists $I_{\omega},I_{\beta}\subseteq \left[n\right]$ such that we have the following nb-arrays:
		$$
		\Gamma(abe):
		\underbrace{
			\begin{array}{c}
			1\\
			1\\
			0
			\end{array}
		}_{c}
		\underbrace{
			\begin{array}{c}
			1\\
			0\\
			1
			\end{array}
		}_{d}
		\underbrace{
			\begin{array}{cccc}
			1 & \dots & \dots & 1\\
			1 & \dots & \dots & 1\\
			1 & \dots & \dots & 1
			\end{array}
		}_{I_{\omega}}
		\underbrace{
			\begin{array}{cccc}
			1 & \dots & \dots & 1\\
			0 & \dots & \dots & 0\\
			0 & \dots & \dots & 0
			\end{array}
		}_{I_{\beta}}
		\begin{array}{c}  
		~\leftarrow ab\text{-row} \\
		~\leftarrow ae\text{-row} \\
		~\leftarrow be\text{-row} 
		\end{array}
		$$
		
		$$
		\Gamma(ace):
		\underbrace{
			\begin{array}{c}
			1\\
			1\\
			0
			\end{array}
		}_{b}
		\underbrace{
			\begin{array}{c}
			1\\
			0\\
			1
			\end{array}
		}_{d}
		\underbrace{
			\begin{array}{cccc}
			1 & \dots & \dots & 1\\
			1 & \dots & \dots & 1\\
			1 & \dots & \dots & 1
			\end{array}
		}_{I_{\omega}}
		\underbrace{
			\begin{array}{cccc}
			1 & \dots & \dots & 1\\
			0 & \dots & \dots & 0\\
			0 & \dots & \dots & 0
			\end{array}
		}_{I_{\beta}}
		\begin{array}{c}  
		~\leftarrow ac\text{-row} \\
		~\leftarrow ae\text{-row} \\
		~\leftarrow ce\text{-row} 
		\end{array}
		$$
	\end{enumerate}
\end{lemma}
\begin{proof}
	1. In Lemma \ref{lem:cIIIcIVii}, we prove that $\overline{ac\ast}-\overline{ad\ast}=n-4$
	and $\overline{ac\ast}=n-2$, so $\overline{ad\ast}=2$. As $p_{11}=2n-5$ and $\overline{ab\ast}=n-2$, we deduce that $\overline{bd\ast}=n-2$. We have completely determined the nb-array of $abd$. 
	
	2. Let $i,j\in \left[n\right]\setminus \{ a,b,c,d \}$. Applying Equation \eqref{eq:rowdiff} to the rows $ab\ast,bc\ast$, we have
	$$
	2=\overline{abi}+\overline{abj}+\overline{cij}-\overline{bci}-\overline{bcj}-\overline{aij}=2+\overline{cij}-\overline{aij}.
	$$ 
	Therefore, we have $\overline{cij}=\overline{aij}$.
	Similarly, applying Equation \eqref{eq:rowdiff} to the rows $ab\ast,ac\ast$, we deduce $\overline{cij}=\overline{bij}$.
	
	3. This follows from 1. and 2., and our previous knowledge of the nb-arrays of $abc,abe$.
	
\end{proof}

We show that $|I_{\omega}|<3$, which means that $n\leq 8$.

\begin{lemma}[Case \labelcref{case:III} and \labelcref{case:IV}(\labelcref{case:IVii})]
	Let $abc$ be a vertex of $\Gamma$ with $\overline{abc}=1$, and let $d\in \left[n\right]\setminus \{a,b,c\}$ be such that we have nb-array
	$$
	\Gamma(abc):
	\underbrace{
		\begin{array}{c}
		1\\
		1\\
		1
		\end{array}
	}_{d}
	\begin{array}{ccccccccccc}
	1 & 1 & \dots & \dots & \dots & \dots & \dots & 1 & 1 & 1 & 1\\
	1 & 1 & \dots & \dots & \dots & \dots & \dots & 1 & 1 & 1 & 1\\
	0 & 0 & \dots & \dots & \dots & \dots & \dots & 0 & 0 & 0 & 0
	\end{array}
	\begin{array}{c}  
	~\leftarrow ab\text{-row} \\
	~\leftarrow ac\text{-row} \\
	~\leftarrow bc\text{-row} 
	\end{array}
	$$
	
	Then $n\leq 8$.
\end{lemma}

\begin{proof}
	We take $d,e,I_{\omega},I_{\beta}$ as in Lemma \ref{lem:cIIIij}.	Suppose $|I_{\omega}|\geq 3$, and let $i,j,k\in I_{\omega}$ be distinct.
	
	First we find the nb-array of $bde$. By assumption, we know $\overline{abc},\overline{abe},\overline{bcd},\overline{bce}$, and by Lemma \ref{lem:cIIIij} 1., we know $\overline{ace},\overline{ade}$. 
	By Lemma \ref{lem:cIIIij} 1., we also know the $bd$-row, and by Lemma \ref{lem:cIIIij} 3., we have determined the $be$-row. We note that the nb-array of $bde$ must then be in Case (IV)(ii), and we have the nb-array
	\begin{equation}\label{eq:IIIiVbde}
	\Gamma(bde):
	\underbrace{
		\begin{array}{c}
		1\\
		1\\
		0
		\end{array}
	}_{a}
	\underbrace{
		\begin{array}{c}
		1\\
		0\\
		1
		\end{array}
	}_{c}
	\underbrace{
		\begin{array}{cccc}
		1 & \dots & \dots & 1\\
		1 & \dots & \dots & 1\\
		1 & \dots & \dots & 1
		\end{array}
	}_{I_{\omega}}
	\underbrace{
		\begin{array}{cccc}
		1 & \dots & \dots & 1\\
		0 & \dots & \dots & 0\\
		0 & \dots & \dots & 0
		\end{array}
	}_{I_{\beta}}
	\begin{array}{c}  
	~\leftarrow bd\text{-row} \\
	~\leftarrow be\text{-row} \\
	~\leftarrow de\text{-row} 
	\end{array}
	\end{equation}
	
	Now consider $bei$. By nb-array \eqref{eq:IIIiVbde}, we know $be\ast$. Applying Equation \eqref{eq:rowdiff},
	\begin{align*}
	\overline{ei\ast}-\overline{bi\ast}&=\frac{n-4}{2}(\overline{aei}+\overline{cei}+\overline{abc}-\overline{abi}-\overline{bci}-\overline{ace})\\
	&=\frac{n-4}{2}(1+1+1-1-0-1)\\
	&=\frac{n-4}{2}
	\end{align*}
	(here we use $\overline{bei}=\overline{cei}$ by Lemma \ref{lem:cIIIij} 3.).
	Therefore $\overline{ei\ast}=n-2$ and we have the nb-array
	\begin{equation}\label{eq:IIIiVbei}
	\Gamma(bei):
	\underbrace{
		\begin{array}{c}
		1\\
		1\\
		1
		\end{array}
	}_{a}
	\underbrace{
		\begin{array}{c}
		1\\
		0\\
		0
		\end{array}
	}_{c}
	\underbrace{
		\begin{array}{c}
		1\\
		1\\
		1
		\end{array}
	}_{d}
	\underbrace{
		\begin{array}{cccc}
		1 & \dots & \dots & 1\\
		\ast & \dots & \dots & \ast\\
		1 & \dots & \dots & 1
		\end{array}
	}_{I_{\omega}\setminus \{i\}}
	\underbrace{
		\begin{array}{cccc}
		1 & \dots & \dots & 1\\
		\ast & \dots & \dots & \ast\\
		0 & \dots & \dots & 0
		\end{array}
	}_{I_{\beta}}
	\begin{array}{c}  
	~\leftarrow ei\text{-row} \\
	~\leftarrow bi\text{-row} \\
	~\leftarrow be\text{-row} 
	\end{array}
	\end{equation}	
	
	Now we find the values of $\overline{bij},\overline{bik}$. By nb-array \eqref{eq:IIIiVbei}, we have 
	$$
	\overline{be\ast}-\overline{bi\ast}=\frac{n-4}{2}(\overline{abe}+\overline{bce}+\overline{aci}-\overline{abi}-\overline{bci}-\overline{ace})=0
	$$
	Therefore 
	\begin{align*}
	\frac{n-4}{2}&=\overline{ei\ast}-\overline{bi\ast}\\
	&=\frac{n-4}{2}(\overline{aei}+\overline{eij}+\overline{abj}-\overline{abi}-\overline{bij}-\overline{aej})\\
	&=\frac{n-4}{2}(1-\overline{bij}).
	\end{align*}
	Therefore $\overline{bij}=0$ and $\overline{bik}=0$, and we have nb-array
	$$
	\Gamma(bei):
	\underbrace{
		\begin{array}{c}
		1\\
		1\\
		1
		\end{array}
	}_{a}
	\underbrace{
		\begin{array}{c}
		1\\
		0\\
		0
		\end{array}
	}_{c}
	\underbrace{
		\begin{array}{c}
		1\\
		1\\
		1
		\end{array}
	}_{d}
	\underbrace{
		\begin{array}{c}
		1\\
		0\\
		1
		\end{array}
	}_{j}
	\underbrace{
		\begin{array}{c}
		1\\
		0\\
		1
		\end{array}
	}_{k}
	\underbrace{
		\begin{array}{cccc}
		1 & \dots & \dots & 1\\
		0 & \dots & \dots & 0\\
		1 & \dots & \dots & 1
		\end{array}
	}_{I_{\omega}\setminus \{i,j,k\}}
	\underbrace{
		\begin{array}{cccc}
		1 & \dots & \dots & 1\\
		\ast & \dots & \dots & \ast\\
		0 & \dots & \dots & 0
		\end{array}
	}_{I_{\beta}}
	\begin{array}{c}  
	~\leftarrow ei\text{-row} \\
	~\leftarrow bi\text{-row} \\
	~\leftarrow be\text{-row} 
	\end{array}
	$$	
	Applying equation \eqref{eq:rowdiff} to the rows $bi\ast,be\ast$, we have 
	$$
	\overline{be\ast}- \overline{bi\ast}=\frac{n-4}{2}(2+\overline{ijk}-\overline{ejk}),
	$$
	which is non-zero. This implies that the nb-array $bei$ is in case \labelcref{case:II}, a contradiction. Therefore $|I_{\omega}|\leq 2$. By looking at the nb-array of $abe$ we see that $|I_{\omega}|=(n-4)/2$, and so $n\leq 8$. 
\end{proof}

\subsection{Removing case (IV)(i)}

\begin{lemma}[Case \labelcref{case:IV}(\labelcref{case:IVi})]
	Let $abc$ be a vertex of $\Gamma$ with $\overline{abc}=1$, and let $d\in \left[n\right]\setminus \{a,b,c\}$ be such that we have nb-array
	$$
	\Gamma(abc):
	\underbrace{\begin{array}{ccccc}
		1 & \dots & \dots & 1 & 1\\
		1 & \dots & \dots & 1 & 1\\
		1 & \dots & \dots & 1 & 1
		\end{array}}_{I_{\omega}}
	\underbrace{\begin{array}{ccccc}
		1 & 1 & \dots & \dots & 1\\
		0 & 0 & \dots & \dots & 0\\
		0 & 0 & \dots & \dots & 0
		\end{array}}_{I_{\beta}}
	\begin{array}{c}  
	~\leftarrow ab\text{-row} \\
	~\leftarrow ac\text{-row} \\
	~\leftarrow bc\text{-row} 
	\end{array}
	$$
	Then $n\leq 6$
\end{lemma}
\begin{proof}
	Note that $|I_{\omega}|=(n-2)/2,|I_{\beta}|=(n-4)/2$ and $p_{11}=2n-5$.
	
	Let $d,e\in I_{\omega}$. Using Equation \eqref{eq:rowdiff}, we see that
	
	$$\frac{n-4}{2}=\overline{ab\ast}-\overline{ac\ast}=\frac{n-4}{2}(\overline{cde}-\overline{bde})$$
	
	Therefore, $\overline{bde}=0,\overline{cde}=1$. Similarly, we can show that $\overline{ade}=0$. 
	
	Let $d\in I_{\omega}$ and consider $acd$. We have $\overline{acd}=1$, and the nb-array
	
	$$
	\Gamma(acd):
	\underbrace{\begin{array}{c}
		1 \\
		1\\
		1
		\end{array}
	}_{b}
	\underbrace{\begin{array}{ccccc}
		1 & \dots & \dots & 1 & 1\\
		0 & \dots & \dots & 0 & 0\\
		1 & \dots & \dots & 1 & 1
		\end{array}}_{I_{\omega}\setminus \{d\}}
	\underbrace{\begin{array}{ccccc}
		0 & 0 & \dots & \dots & 0\\
		* & * & \dots & \dots & *\\
		* & * & \dots & \dots & *
		\end{array}}_{I_{\beta}}
	\begin{array}{c}  
	~\leftarrow ac\text{-row} \\
	~\leftarrow ad\text{-row} \\
	~\leftarrow cd\text{-row} 
	\end{array}
	$$
	
	As $p_{11}=2n-5$, we must have $\overline{ade}=\overline{cde}=1$ for all $e\in I_{\beta}$. Therefore, $acd$ is in case \labelcref{case:IV}. Looking at the possible arrays for case \labelcref{case:IV}, we see that the nb-array for $acd$ can only be valid when $|I_{\beta}|\leq 1$, so $n\leq 6$.
\end{proof}

\subsection{Removing case \labelcref{case:VI}(\labelcref{case:VIii})}\label{sec:rmVIii}

\begin{lemma}[Case \labelcref{case:VI}(\labelcref{case:VIii})]
	Let $abc$ be a vertex of $\Gamma$ with $\overline{abc}=1$, and let $d,e\in \left[n\right]\setminus \{a,b,c\}$ be such that we have nb-array
	$$
	\Gamma(abc):
	\underbrace{\begin{array}{c}
		1\\
		1\\
		0
		\end{array}}_{d}
	\underbrace{\begin{array}{c}
		0\\
		0\\
		1
		\end{array}}_{e}
	\underbrace{\begin{array}{ccccc}
		1 & 1 & \dots & \dots & 1\\
		1 & 1 & \dots & \dots & 1\\
		1 & 1 & \dots & \dots & 1
		\end{array}}_{I_{\omega}}
	\underbrace{\begin{array}{cccc}
		0 & \dots & \dots & 0\\
		0 & \dots & \dots & 0\\
		0 & \dots & \dots & 0
		\end{array}}_{I_{\beta}}
	\begin{array}{c}  
	~\leftarrow ab\text{-row} \\
	~\leftarrow ac\text{-row} \\
	~\leftarrow bc\text{-row} 
	\end{array}
	$$
	Then $n\leq 14$
\end{lemma}
\begin{proof}
	Suppose $n>14$, and let $t_{VI}=\overline{ab\ast}-1$. Then $p_{11}=\overline{\Gamma(abc)}=3t_{VI}$ and so $t_{VI}\geq (2n-7)/3$. We  then observe that when $n> 14$, $t_{VI}\geq (2n-7)/3> n/2$.
	
	Using Equation \eqref{eq:rowdiff}, we see that
	
	$$0=\overline{ab\ast}-\overline{bc\ast}=\frac{n-4}{2}(\overline{cde}-\overline{ade})$$
	
	Therefore, $\overline{ade}=\overline{cde}$. Similarly, we can show that $\overline{bde}=\overline{ade}=\overline{cde}$.
	
	Then we consider $abd$. We have $\overline{abd}=1$ and 
	
	\begin{eqnarray*}
		\overline{ab\ast}-\overline{ad\ast}&=&\frac{n-4}{2}(\overline{abc}+\overline{abe}+\overline{cde}-\overline{acd}-\overline{ade}-\overline{bce})\\
		&=&\frac{n-4}{2}(\overline{cde}-\overline{ade}-1)\\
		&=&-(n-4)/2
	\end{eqnarray*}
	
	Therefore, $\overline{ad\ast}>n/2+(n-4)/2=n-2$ when $n>14$. 
	
\end{proof}

\subsection{Removing case \labelcref{case:VI}(\labelcref{case:VIiv})}
\begin{lemma}[Case \labelcref{case:VI}(\labelcref{case:VIiv})]
	Let $abc$ be a vertex of $\Gamma$ with $\overline{abc}=1$, and let $d,e,f\in \left[n\right]\setminus \{a,b,c\}$ be such that we have nb-array
	$$
	\Gamma(abc):
	\underbrace{\begin{array}{c}
		1\\
		1\\
		0
		\end{array}}_{d}
	\underbrace{\begin{array}{c}
		1\\
		0\\
		1
		\end{array}}_{e}
	\underbrace{\begin{array}{c}
		0\\
		1\\
		1
		\end{array}}_{f}
	\underbrace{\begin{array}{cccc}
		1 & \dots & \dots & 1\\
		1 & \dots & \dots & 1\\
		1 & \dots & \dots & 1
		\end{array}}_{I_{\omega}}
	\underbrace{\begin{array}{cccc}
		0 & \dots & \dots & 0\\
		0 & \dots & \dots & 0\\
		0 & \dots & \dots & 0
		\end{array}}_{I_{\beta}}
	\begin{array}{c}  
	~\leftarrow ab\text{-row} \\
	~\leftarrow ac\text{-row} \\
	~\leftarrow bc\text{-row} 
	\end{array}
	$$
	Then $n\leq 8$
\end{lemma}
\begin{proof}
	Note that we have a lower bound on the size of $I_{\omega}$. 
	
	Using Equation \eqref{eq:rowdiff}, we see that
	
	$$0=\overline{ab\ast}-\overline{ac\ast}=\frac{n-4}{2}(1+\overline{cde}-\overline{bde})$$
	
	Therefore, $\overline{cde}=0,\overline{bde}=1$. Similarly, we can show that $\overline{ade}=1$. We can consider pairs $d,f$ or $e,f$, or use symmetry to see that $\overline{adf}=\overline{cdf}=1,\overline{bdf}=0$ and $\overline{bef}=\overline{cef}=1,\overline{aef}=0$
	
	Let $g\in I_{\omega}\cup I_{\beta}$ Using Equation \eqref{eq:rowdiff}, we see that
	
	$$0=\overline{ab\ast}-\overline{ac\ast}=\frac{n-4}{2}((\overline{abg}-\overline{acg})+(\overline{cdg}-\overline{bdg}))$$
	$$
	0=\overline{ab\ast}-\overline{bc\ast}=\frac{n-4}{2}(1+(\overline{abg}-\overline{bcg})+(\overline{cdg}-\overline{adg})).
	$$
	
	Therefore, $\overline{adg}=1,\overline{bdg}=0=\overline{adg}$.
	
	Now consider $abd$. We know $\overline{abd}=1$, and it has the nb-array

	$$
	\Gamma(abd):
	\underbrace{\begin{array}{c}
		1\\
		1\\
		0
		\end{array}}_{c}
	\underbrace{\begin{array}{c}
		1\\
		1\\
		1
		\end{array}}_{e}
	\underbrace{\begin{array}{c}
		0\\
		1\\
		0
		\end{array}}_{f}
	\underbrace{\begin{array}{cccc}
		1 & \dots & \dots & 1\\
		1 & \dots & \dots & 1\\
		0 & \dots & \dots & 0
		\end{array}}_{I_{\omega}}
	\underbrace{\begin{array}{cccc}
		0 & \dots & \dots & 0\\
		1 & \dots & \dots & 1\\
		0 & \dots & \dots & 0
		\end{array}}_{I_{\beta}}
	\begin{array}{c}  
	~\leftarrow ab\text{-row} \\
	~\leftarrow ad\text{-row} \\
	~\leftarrow bd\text{-row} 
	\end{array}
	$$
	This array is in case \labelcref{case:II}, so $n\leq 8$
\end{proof}

\subsection{Removing case \labelcref{case:VI}(\labelcref{case:VIi})}
\begin{lemma}[Case \labelcref{case:VI}(\labelcref{case:VIi})]
	Let $abc$ be a vertex of $\Gamma$ with $\overline{abc}=1$ be such that we have nb-array
	
	$$
	\Gamma(abc):
	\underbrace{\begin{array}{ccccccc}
		1 & 1 & 1 & 1 & \dots & \dots & 1\\
		1 & 1 & 1 & 1 & \dots & \dots & 1\\
		1 & 1 & 1 & 1 & \dots & \dots & 1
		\end{array}}_{I_{\omega}}
	\underbrace{\begin{array}{cccc}
		0 & \dots & \dots & 0\\
		0 & \dots & \dots & 0\\
		0 & \dots & \dots & 0
		\end{array}}_{I_{\beta}}
	\begin{array}{c}  
	~\leftarrow ab\text{-row} \\
	~\leftarrow ac\text{-row} \\
	~\leftarrow bc\text{-row} 
	\end{array}
	$$
	Then $n\leq 14$.
\end{lemma}
\begin{proof}
	Suppose $n>14$, and let $t_{VI}=\overline{ab\ast}-1$. Then $p_{11}=\overline{\Gamma(abc)}=3t_{VI}$ and so $t_{VI}\geq (2n-7)/3$. We  then observe that when $n> 14$, $t_{VI}\geq (2n-7)/3> n/2$.
	
	Let $d,e\in I_{\omega}\cup I_{\beta}$. Using Equation \eqref{eq:rowdiff}, we see that
	
	$$0=\overline{ab\ast}-\overline{ac\ast}=\frac{n-4}{2}(\overline{cde}-\overline{bde})$$
	
	Therefore, $\overline{cde}=\overline{bde}$. Similarly, we can show that $\overline{ade}=\overline{cde}=\overline{bde}$.
	
	Now fix $d\in I_{\omega}$ and let $e\in I_{\omega}\cup I_{\beta}$. Then $\overline{abd}=1$ and
	
	$$\overline{ab\ast}-\overline{ad\ast}=\frac{n-4}{2}((\overline{abe}-\overline{bce})+(\overline{cde}-\overline{ade}))=0$$
	
	$$\overline{ab\ast}-\overline{bd\ast}=\frac{n-4}{2}((\overline{abe}-\overline{ace})+(\overline{cde}-\overline{bde}))=0$$
	
	So $abd$ is in case \labelcref{case:VI} and has the following nb-array.
	
	$$
	\Gamma(abd):
	\underbrace{\begin{array}{c}
		1\\
		1\\
		1
		\end{array}}_{c}
	\underbrace{\begin{array}{cccc}
		1 & \dots & \dots & 1\\
		* & \dots & \dots & *\\
		* & \dots & \dots & *
		\end{array}}_{I_{\omega}}
	\underbrace{\begin{array}{cccc}
		0 & \dots & \dots & 0\\
		* & \dots & \dots & *\\
		* & \dots & \dots & *
		\end{array}}_{I_{\beta}}
	\begin{array}{c}  
	~\leftarrow ab\text{-row} \\
	~\leftarrow ad\text{-row} \\
	~\leftarrow bd\text{-row} 
	\end{array}
	$$
	
	Using the fact that $\overline{ade}=\overline{bde}$ for any $e\in I_{\omega}\cup I_{\beta}$, we see that the nb-array of $abd$ is in case \labelcref{case:VI}(\labelcref{case:VIii}) or case \labelcref{case:VI}(\labelcref{case:VIi}). 
	
	When $n>14$, we must be in case \labelcref{case:VI}(\labelcref{case:VIi}) (see section \ref{sec:rmVIii}). Therefore we have the following nb-array. 
	$$
	\Gamma(abd):
	\underbrace{\begin{array}{c}
		1\\
		1\\
		1
		\end{array}}_{c}
	\underbrace{\begin{array}{cccc}
		1 & \dots & \dots & 1\\
		1 & \dots & \dots & 1\\
		1 & \dots & \dots & 1
		\end{array}}_{I_{\omega}}
	\underbrace{\begin{array}{cccc}
		0 & \dots & \dots & 0\\
		0 & \dots & \dots & 0\\
		0 & \dots & \dots & 0
		\end{array}}_{I_{\beta}}
	\begin{array}{c}  
	~\leftarrow ab\text{-row} \\
	~\leftarrow ad\text{-row} \\
	~\leftarrow bd\text{-row} 
	\end{array}
	$$
	
	We can then continue iterating this process. For any distinct elements $d,e,f\in I_{\omega}\cup \{a,b,c\}$ the nb-array is as follows.
	
	$$
	\Gamma(def):
	\underbrace{\begin{array}{cccc}
		1 & \dots & \dots & 1\\
		1 & \dots & \dots & 1\\
		1 & \dots & \dots & 1
		\end{array}}_{I_{\omega}\cup \{a,b,c\}\setminus \{d,e,f\}}
	\underbrace{\begin{array}{cccc}
		0 & \dots & \dots & 0\\
		0 & \dots & \dots & 0\\
		0 & \dots & \dots & 0
		\end{array}}_{I_{\beta}}
	\begin{array}{c}  
	~\leftarrow de\text{-row} \\
	~\leftarrow df\text{-row} \\
	~\leftarrow ef\text{-row} 
	\end{array}
	$$
	
	Now suppose there is $d\in I_{\omega}\cup \{a,b,c\},g,h\in I_{\beta}$ such that $\overline{dgh}=1$. Then $dgh$ has the following nb-array.

	$$
	\Gamma(dgh):
	\underbrace{\begin{array}{cccc}
		0 & \dots & \dots & 0\\
		0 & \dots & \dots & 0\\
		* & \dots & \dots & *
		\end{array}}_{I_{\omega}\cup \{b,c\}}
	\underbrace{\begin{array}{cccc}
		* & \dots & \dots & *\\
		* & \dots & \dots & *\\
		* & \dots & \dots & *
		\end{array}}_{I_{\beta}\setminus \{g,h\}}
	\begin{array}{c}  
	~\leftarrow dg\text{-row} \\
	~\leftarrow dh\text{-row} \\
	~\leftarrow gh\text{-row} 
	\end{array}
	$$
	
	As $n>14$ $|I_{\omega}|>(n-3)/2,|I_{\omega}|+|I_{\beta}|=n-3$, we see that $|I_{\beta}\setminus \{g,h\}|<(n-4)/2$. We observe that $\overline{dg\ast}$, $\overline{dh\ast}$ is at most $|I_{\beta}\setminus \{g,h\}|+1$, and therefore $\overline{dg\ast}-\overline{dh\ast}=0$. Looking back at the nb-array of $dgh$, we have $p_{11}<(n-3)+2(n-4)/2=2n-7$, contradicting our assumption.
	
	Therefore there is no vertex $dgh$ such that $d\in I_{\omega}\cup \{a,b,c\},g,h\in I_{\beta}$ and $\overline{dgh}=1$.  
	
	Now suppose there is $g,h,i\in I_{\beta}$ such that $\overline{ghi}=1$. Then $ghi$ has the following nb-array.
	
	$$
	\Gamma(ghi):
	\underbrace{\begin{array}{cccc}
		0 & \dots & \dots & 0\\
		0 & \dots & \dots & 0\\
		0 & \dots & \dots & 0
		\end{array}}_{I_{\omega}\cup \{a,b,c\}}
	\underbrace{\begin{array}{cccc}
		* & \dots & \dots & *\\
		* & \dots & \dots & *\\
		* & \dots & \dots & *
		\end{array}}_{I_{\beta}\setminus \{g,h,i\}}
	\begin{array}{c}  
	~\leftarrow gh\text{-row} \\
	~\leftarrow gi\text{-row} \\
	~\leftarrow hi\text{-row} 
	\end{array}
	$$
	
	We see that $\overline{gh\ast},\overline{gi\ast},\overline{hi\ast}$ are at most $|I_{\beta}|-3< (n-4)/2$. Therefore, $p_{11}<3(n-4)/2<2n-7$, a contradiction. 
	
	Therefore there is no vertex $ghi$ such that $g,h,i\in I_{\beta}$ and $\overline{ghi}=1$.
	
	We have shown that a vertex $def$ is such that $\overline{def}=1$ if and only if $d,e,f\in I_{\omega}\cup \{a,b,c\}$. In other words,
	
	$$X_1=\{def\in V(\Gamma):def\subset I_{\omega}\}.$$
	
	For $g,h,i\in I_{\beta}$ or $g,h\in I_{\beta},i\in I_{\omega}\cup \{a,b,c\}$, $ghi$ is not adjacent to any vertex in $X_1$. For $g\in I_{\beta},h,i\in I_{\omega}\cup \{a,b,c\}$, $ghi$ is adjacent to $|I_{\omega}\cup \{a,b,c\}|-2$ vertices in $X_1$ ($dhi$, for $d\in I_{\omega}\cup \{a,b,c\}\setminus \{h,i\}$) This is not a regular set unless $|I_{\beta}|=1$. In this case, we have found a $\lambda_1$-equitable 2-partition.    
\end{proof}

We summarise the results of this section in the following table. The first column of the table give the numeral of a difference tuple. The second column contains the numeral given to a possible nb-array with the corresponding difference tuple. The final column gives a upper bound on $n$ for which the corresponding nb-array can be present in a $\lambda_2$-equitable 2-partition. 

\begin{table}[ht!]
	\begin{center}
		$\begin{array}{c|c|c}
		\text{tuple} & \text{nb-array} & \text{max } n\\
		\hline 
		\labelcref{case:I}     &          & 6\\
		\labelcref{case:II}    &          & 8\\
		\labelcref{case:III}   &          & 8\\
		\labelcref{case:IV}    & (\labelcref{case:IVi})      & 6\\
		\labelcref{case:IV}    & (\labelcref{case:IVii})     & 8\\
		\labelcref{case:V}     & (\labelcref{case:Vi})     & \text{none}\\
		\labelcref{case:V}     & (\labelcref{case:Vii})     & \text{none}\\
		\labelcref{case:VI}    & (\labelcref{case:VIi})     & 14\\
		\labelcref{case:VI}    & (\labelcref{case:VIii})     & 14\\
		\labelcref{case:VI}    & (\labelcref{case:VIiii})    & \text{none}\\
		\labelcref{case:VI}    & (\labelcref{case:VIiv})     & 8	
		\end{array}$
	\end{center}
	\caption{Bounds on $n$ given the existence of a vertex with specified nb-array.}\label{tab:posnba}
\end{table}

\section{Classification of $\lambda_2$ partitions for $n>14$}\label{sec:ng14}

Given the results of Section \ref{sec:Jt2}, we now assume $n>14$. This means that the only possible nb-arrays are \labelcref{case:VI}(\labelcref{case:VIiii}),\labelcref{case:V}(\labelcref{case:Vi}) and \labelcref{case:V}(\labelcref{case:Vii}).

\subsection{Case \labelcref{case:V}(\labelcref{case:Vii})}\label{sec:caseVii} 

In this section, we prove the following.

\begin{thm}\label{thm:Vii}
	Let $abc$ be a vertex of $\Gamma$ with $\overline{abc}=1$ be such that we have nb-array as in \labelcref{case:V}(\labelcref{case:Vii}).
	
	Then the partition $X$ can be constructed from a matching on the elements of $\left[n\right]$, which corresponds to an instance of a partition $\Pi_3$.
\end{thm}

To do this, we start by characterising the quotient matrices when we are in this case.

\begin{lemma}
	Let $abc$ be a vertex of $\Gamma$ with $\overline{abc}=1$ be such that we have nb-array
	$$
	\Gamma(abc):
	\underbrace{\begin{array}{c}
		1\\
		0\\
		0
		\end{array}}_{d}
	\underbrace{\begin{array}{c}
		0\\
		1\\
		0
		\end{array}}_{e}
	\underbrace{\begin{array}{cccc}
		1 & \dots & \dots & 1\\
		1 & \dots & \dots & 1\\
		1 & \dots & \dots & 1
		\end{array}}_{I_\omega}
	\underbrace{\begin{array}{cccc}
		1 & \dots & \dots & 1\\
		1 & \dots & \dots & 1\\
		0 & \dots & \dots & 0
		\end{array}}_{I_{\gamma}}
	\underbrace{\begin{array}{cccc}
		0 & \dots & \dots & 0\\
		0 & \dots & \dots & 0\\
		0 & \dots & \dots & 0
		\end{array}}_{I_{\beta}}
	\begin{array}{c}  
	~\leftarrow ab\text{-row} \\
	~\leftarrow ac\text{-row} \\
	~\leftarrow bc\text{-row} 
	\end{array}
	$$
	Then the quotient matrix for the partition is 
	\begin{equation*}
	\left(\begin{matrix}
	5n/2-13 & n/2+4\\
	3n/2-1 & 3n/2-3
	\end{matrix}\right)
	\end{equation*}
	
\end{lemma}
\begin{proof}
	By using Equation \eqref{eq:rowdiff}, we see that 
	\begin{align*}
	\overline{ade}&=\overline{bde}=\overline{cde} \\
	\overline{adg}&=\overline{cdg}=0, \overline{bdg}=1 \qquad (for\quad g\in I_{\omega}\cup I_{\beta})\\
	\overline{adg}&=\overline{bdg}=1,\overline{cdg}=0\qquad (for\quad g\in I_{\gamma})
	\end{align*}
	Consider $abd$. We have $\overline{abd}=1$, and it has the following nb-array.
	
	$$
	\Gamma(abd):
	\underbrace{\begin{array}{c}
		0\\
		1\\
		0
		\end{array}}_{c}
	\underbrace{\begin{array}{c}
		*\\
		0\\
		*
		\end{array}}_{e}
	\underbrace{\begin{array}{cccc}
		1 & \dots & \dots & 1\\
		1 & \dots & \dots & 1\\
		1 & \dots & \dots & 1
		\end{array}}_{I_\gamma}
	\underbrace{\begin{array}{cccc}
		1 & \dots & \dots & 1\\
		1 & \dots & \dots & 1\\
		0 & \dots & \dots & 0
		\end{array}}_{I_{\omega}\setminus \{d\}}
	\underbrace{\begin{array}{cccc}
		1 & \dots & \dots & 1\\
		0 & \dots & \dots & 0\\
		0 & \dots & \dots & 0
		\end{array}}_{I_{\beta}}
	\begin{array}{c}  
	~\leftarrow bd\text{-row} \\
	~\leftarrow ab\text{-row} \\
	~\leftarrow ad\text{-row} 
	\end{array}
	$$
	
	Note that the $\ast$ are equal here. By using Equation \eqref{eq:rowdiff}, we also see that $abd$ is in case \labelcref{case:V}. Observing the $c$-column, we deduce that $|I_{\beta}|=1$ and the $e$-column contains only zeros. Then $p_{11}=5n/2-13$. We deduce the remaining values of the quotient matrix from Lemma \ref{lem:theta}.
\end{proof}

Therefore, we can assume we have a vertex $\hat{a}bc$ such that $\overline{\hat{a}bc}=1$ and the nb-array
$$
\Gamma(\hat{a}bc):
\underbrace{\begin{array}{c}
	0\\
	0\\
	0
	\end{array}}_{a}
\underbrace{\begin{array}{c}
	0\\
	1\\
	0
	\end{array}}_{\hat{b}}
\underbrace{\begin{array}{c}
	1\\
	0\\
	0
	\end{array}}_{\hat{c}}
\underbrace{\begin{array}{cccc}
	1 & \dots & \dots & 1\\
	1 & \dots & \dots & 1\\
	1 & \dots & \dots & 1
	\end{array}}_{I_\omega}
\underbrace{\begin{array}{cccc}
	1 & \dots & \dots & 1\\
	1 & \dots & \dots & 1\\
	0 & \dots & \dots & 0
	\end{array}}_{I_{\gamma}}
\begin{array}{c}  
~\leftarrow \hat{a}b\text{-row} \\
~\leftarrow \hat{a}c\text{-row} \\
~\leftarrow bc\text{-row} 
\end{array}
$$
Note that we have been careful to pick $a$ as the index of the unique column of all 0s and $\hat{a}$ as the index in $\hat{a}bc$ such that $\overline{bc\ast}$ is minimum. For the moment, we will let 
\begin{eqnarray*}
	W&=&I_{\omega}\cup \{\hat{a},\hat{b},\hat{c}\},\\
	U&=&I_{\gamma}\cup \{a,b,c\}.
\end{eqnarray*}
Then we have the following using Equation \eqref{eq:rowdiff}.

\begin{align*}
&\overline{\hat{a}\hat{b}\hat{c}}=\overline{b\hat{b}\hat{c}}=\overline{\hat{b}c\hat{c}} \\
&\overline{a\hat{a}\hat{c}}=\overline{ac\hat{c}}=0, \overline{ab\hat{c}}=1 \\
&\overline{a\hat{a}\hat{b}}=\overline{ab\hat{b}}=0,\overline{ac\hat{b}}=1 \\
&\overline{\hat{a}\hat{c}g}=\overline{b\hat{c}g}=1,\overline{c\hat{c}g}=0\qquad &(for\quad g\in I_{\gamma})\\
&\overline{\hat{a}\hat{b}g}=\overline{c\hat{b}g}=1,\overline{b\hat{b}g}=0 \qquad &(for\quad g\in I_{\gamma}).\\
&\overline{a\hat{a}g}=\overline{abg}=\overline{acg}\qquad &(for\quad g\in I_{\gamma}).\\
&\overline{\hat{bgh}}=\overline{cgh}=0,\overline{\hat{a}gh}=1\qquad &(for\quad g,h\in I_{\gamma})\\
&\overline{\hat{a}\hat{g}h}=\overline{b\hat{g}h}=\overline{c\hat{g}h}\qquad &(for\quad \hat{g}\in I_{\omega},h\in I_{\gamma}).
\end{align*}
Let $d\in I_{\gamma}$. Then we have the following nb-array.
$$
\Gamma(\hat{a}bd):
\underbrace{\begin{array}{c}
	0\\
	*\\
	*
	\end{array}}_{a}
\underbrace{\begin{array}{c}
	0\\
	1\\
	0
	\end{array}}_{\hat{b}}
\underbrace{\begin{array}{c}
	1\\
	1\\
	1
	\end{array}}_{\hat{c}}
\underbrace{\begin{array}{c}
	1\\
	1\\
	0
	\end{array}}_{c}
\underbrace{\begin{array}{cccc}
	1 & \dots & \dots & 1\\
	* & \dots & \dots & *\\
	* & \dots & \dots & *
	\end{array}}_{I_\omega}
\underbrace{\begin{array}{cccc}
	1 & \dots & \dots & 1\\
	1 & \dots & \dots & 1\\
	0 & \dots & \dots & 0
	\end{array}}_{I_{\gamma}\setminus \{d\}}
\begin{array}{c}  
~\leftarrow \hat{a}b\text{-row} \\
~\leftarrow \hat{a}d\text{-row} \\
~\leftarrow bd\text{-row} 
\end{array}
$$

Note that any two $\ast$ in the same column must be the same. By applying Equation \eqref{eq:rowdiff} and noting that there is a column containing exactly one 1, we see that the nb-array of $\hat{a}bd$ must be in case \labelcref{case:V}(\labelcref{case:Vii}). So the $a$-column must be all 0, and by considering $p_{11}$ we deduce that there exists a unique element $\hat{d}\in I_{\omega}$ such that $\overline{\hat{a}d\hat{d}}=0=\overline{bd\hat{d}}$. If there was another $e\in I_{\gamma}$ with $\overline{\hat{a}e\hat{d}}=0=\overline{be\hat{d}}$, then $\overline{\hat{a}b\hat{d}}=1$ and has $d$ and $e$-columns equal and containing exactly one entry of 1. However, there are no nb-arrays in the cases \labelcref{case:I,case:II,case:III,case:IV,case:V,case:VI} such that this occurs, giving a contradiction. Therefore we have a matching of $U$ with $W$, given by
\begin{equation*}M=\{(a,\hat{a}),(b,\hat{b}),(c,\hat{c})\}\cup \{(d,\hat{d}):d\in I_{\gamma}\}.
\end{equation*}

Using this matching, we can see the following.
\begin{lemma}\label{lem:cViipr}
	We have the following: 
	\begin{enumerate}
		\item For $d\in I_{\gamma}\cup \{c\}$ we have $\overline{\hat{a}bd}=1$ and nb-array 
		$$
		\Gamma(\hat{a}bd):
		\underbrace{\begin{array}{c}
			0\\
			0\\
			0
			\end{array}}_{a}
		\underbrace{\begin{array}{c}
			0\\
			1\\
			0
			\end{array}}_{\hat{b}}
		\underbrace{\begin{array}{c}
			1\\
			0\\
			0
			\end{array}}_{\hat{d}}
		\underbrace{\begin{array}{cccc}
			1 & \dots & \dots & 1\\
			1 & \dots & \dots & 1\\
			1 & \dots & \dots & 1
			\end{array}}_{I_\omega\cup \{\hat{c}\}\setminus \{\hat{d}\}}
		\underbrace{\begin{array}{cccc}
			1 & \dots & \dots & 1\\
			1 & \dots & \dots & 1\\
			0 & \dots & \dots & 0
			\end{array}}_{I_{\gamma}\cup \{c\}\setminus \{d\}}
		\begin{array}{c}  
		~\leftarrow \hat{a}b\text{-row} \\
		~\leftarrow \hat{a}d\text{-row} \\
		~\leftarrow bd\text{-row} 
		\end{array}
		$$ 
		\item For $d\in I_{\gamma}\cup \{c\}$ we have $\overline{\hat{a}b\hat{d}}=1$ and nb-array$$
		\Gamma(\hat{a}b\hat{d}):
		\underbrace{\begin{array}{c}
			0\\
			0\\
			0
			\end{array}}_{\hat{b}}
		\underbrace{\begin{array}{c}
			0\\
			1\\
			0
			\end{array}}_{\hat{a}}
		\underbrace{\begin{array}{c}
			1\\
			0\\
			0
			\end{array}}_{d}
		\underbrace{\begin{array}{cccc}
			1 & \dots & \dots & 1\\
			1 & \dots & \dots & 1\\
			0 & \dots & \dots & 0
			\end{array}}_{I_\omega\cup \{\hat{c}\}\setminus \{\hat{d}\}}
		\underbrace{\begin{array}{cccc}
			1 & \dots & \dots & 1\\
			1 & \dots & \dots & 1\\
			1 & \dots & \dots & 1
			\end{array}}_{I_{\gamma}\cup \{c\}\setminus \{d\}}
		\begin{array}{c}  
		~\leftarrow \hat{a}b\text{-row} \\
		~\leftarrow b\hat{d}\text{-row} \\
		~\leftarrow \hat{b}\hat{d}\text{-row} 
		\end{array}
		$$
		\item For $d\in I_{\gamma}\cup \{a\}$ we have $\overline{\hat{d}bc}=1$ and nb-array
		$$
		\Gamma(\hat{d}bc):
		\underbrace{\begin{array}{c}
			0\\
			0\\
			0
			\end{array}}_{d}
		\underbrace{\begin{array}{c}
			0\\
			1\\
			0
			\end{array}}_{\hat{b}}
		\underbrace{\begin{array}{c}
			1\\
			0\\
			0
			\end{array}}_{\hat{c}}
		\underbrace{\begin{array}{cccc}
			1 & \dots & \dots & 1\\
			1 & \dots & \dots & 1\\
			1 & \dots & \dots & 1
			\end{array}}_{I_\omega\cup \{\hat{a}\}\setminus \{\hat{d}\}}
		\underbrace{\begin{array}{cccc}
			1 & \dots & \dots & 1\\
			1 & \dots & \dots & 1\\
			0 & \dots & \dots & 0
			\end{array}}_{I_{\gamma}\cup \{a\}\setminus \{d\}}
		\begin{array}{c}  
		~\leftarrow b\hat{d}\text{-row} \\
		~\leftarrow c\hat{d}\text{-row} \\
		~\leftarrow bc\text{-row} 
		\end{array}
		$$ 
	\end{enumerate}
\end{lemma}
\begin{proof}
	This follows from the definition of the matching $M$ and similar arguments to the above discussion.
\end{proof}

\begin{cor}\label{cor:genVii}
	For all distinct $d,e,f\in U$, we have $\overline{\hat{d}ef}=1,\overline{\hat{d}e\hat{f}}=1$, and the following nb-arrays:
	\begin{enumerate}
		\item $$\Gamma(\hat{d}ef):
		\underbrace{\begin{array}{c}
			0\\
			0\\
			0
			\end{array}}_{d}
		\underbrace{\begin{array}{c}
			0\\
			1\\
			0
			\end{array}}_{\hat{e}}
		\underbrace{\begin{array}{c}
			1\\
			0\\
			0
			\end{array}}_{\hat{f}}
		\underbrace{\begin{array}{cccc}
			1 & \dots & \dots & 1\\
			1 & \dots & \dots & 1\\
			1 & \dots & \dots & 1
			\end{array}}_{W\setminus \{\hat{d},\hat{e},\hat{f} \}}
		\underbrace{\begin{array}{cccc}
			1 & \dots & \dots & 1\\
			1 & \dots & \dots & 1\\
			0 & \dots & \dots & 0
			\end{array}}_{U\setminus \{d,e,f\}}
		\begin{array}{c}  
		~\leftarrow \hat{d}e\text{-row} \\
		~\leftarrow \hat{d}f\text{-row} \\
		~\leftarrow ef\text{-row} 
		\end{array}$$
		\item  $$
		\Gamma(\hat{d}e\hat{f}):
		\underbrace{\begin{array}{c}
			0\\
			0\\
			0
			\end{array}}_{\hat{e}}
		\underbrace{\begin{array}{c}
			0\\
			1\\
			0
			\end{array}}_{d}
		\underbrace{\begin{array}{c}
			1\\
			0\\
			0
			\end{array}}_{f}
		\underbrace{\begin{array}{cccc}
			1 & \dots & \dots & 1\\
			1 & \dots & \dots & 1\\
			0 & \dots & \dots & 0
			\end{array}}_{W\setminus \{\hat{d},\hat{e},\hat{f}\}}
		\underbrace{\begin{array}{cccc}
			1 & \dots & \dots & 1\\
			1 & \dots & \dots & 1\\
			1 & \dots & \dots & 1
			\end{array}}_{U\setminus \{d,e,f\}}
		\begin{array}{c}  
		~\leftarrow \hat{d}e\text{-row} \\
		~\leftarrow e\hat{f}\text{-row} \\
		~\leftarrow \hat{d}\hat{f}\text{-row} 
		\end{array}	$$	
	\end{enumerate}
\end{cor}
\begin{proof}
	This follows from the repeated application of Lemma \ref{lem:cViipr}. If $n>8$, we can we can always apply the lemma consecutively, starting from $\hat{a}bc$ to attain the nb-array of $\hat{d}ef,\hat{d}e\hat{f}$.
\end{proof}

\begin{proof}[Proof of Theorem \ref{thm:Vii}]
	Any vertex not considered directly in Corrolary \ref{cor:genVii} is of the form $def,d\hat{d}e,\hat{d}\hat{e}\hat{f}$ or $d\hat{d}\hat{f}$ for distinct elements $d,e,f\in U$. The first two appear in the nb-array of $\hat{d}ef$, and the last two appear in the nb-array of $\hat{d}e\hat{f}$. By observing these nb-arrays, we see that all have value 0. This proves that the partition is exactly the partition $\Pi_3$, with matching $M$ in Section \ref{sec:knJn3}.
\end{proof}

\subsection{Case \labelcref{case:V}(\labelcref{case:Vi})}

In this section, we prove the following.

\begin{thm}\label{thm:Vi}
	Let $abc$ be a vertex of $\Gamma$ with $\overline{abc}=1$ be such that we have nb-array as in \labelcref{case:V}(\labelcref{case:Vi}).
	
	Then the partition $X$ can be constructed from a matching on the elements of $\left[n\right]$, which corresponds to an instance of a partition $\Pi_1$.
\end{thm}

First we characterise the possible quotient matrices in this case.

\begin{lemma}
	Let $abc$ be a vertex of $\Gamma$ with $\overline{abc}=1$ be such that we have nb-array
	$$
	\Gamma(abc):
	\underbrace{\begin{array}{cccc}
		1 & \dots & \dots & 1\\
		1 & \dots & \dots & 1\\
		1 & \dots & \dots & 1
		\end{array}}_{I_{\omega}}
	\underbrace{\begin{array}{cccc}
		1 & \dots & \dots & 1\\
		1 & \dots & \dots & 1\\
		0 & \dots & \dots & 0
		\end{array}}_{I_{\gamma}}
	\underbrace{\begin{array}{cccc}
		0 & \dots & \dots & 0\\
		0 & \dots & \dots & 0\\
		0 & \dots & \dots & 0
		\end{array}}_{I_{\beta}}
	\begin{array}{c}  
	~\leftarrow ab\text{-row} \\
	~\leftarrow ac\text{-row} \\
	~\leftarrow bc\text{-row} 
	\end{array}
	$$
	Then the quotient matrix for the partition is 
	\begin{equation*}
	\left(\begin{matrix}
	5n/2-7 & n/2-2\\
	3n/2 & 3n/2-9
	\end{matrix}\right)
	\end{equation*}
\end{lemma}
\begin{proof}
	By using Equation \eqref{eq:rowdiff}, we see that 
	\begin{align*}
	&\overline{bde}=\overline{cde}=1, \overline{ade}=0  \qquad &(for\quad d,e\in I_{\omega}\cup I_{\beta})\\
	&\overline{bde}=\overline{cde}=0, \overline{ade}=1  \qquad &(for\quad d,e\in I_{\gamma})\\
	&\overline{ade}=\overline{bde}=\overline{cde} \qquad &(for\quad d\in I_{\omega}\cup I_{\beta},e\in I_{\gamma})
	\end{align*}
	
	Let $d \in I_{\omega}$ We have $\overline{acd}=1$ and the following nb-array.
	
	$$
	\Gamma(acd):
	\underbrace{\begin{array}{c}
		1\\
		1\\
		1
		\end{array}}_{b}
	\underbrace{\begin{array}{cccc}
		1 & \dots & \dots & 1\\
		1 & \dots & \dots & 1\\
		0 & \dots & \dots & 0
		\end{array}}_{I_{\omega}\setminus \{d\}}
	\underbrace{\begin{array}{cccc}
		0 & \dots & \dots & 0\\
		1 & \dots & \dots & 1\\
		0 & \dots & \dots & 0
		\end{array}}_{I_{\beta}}
	\underbrace{\begin{array}{cccc}
		1 & \dots & \dots & 1\\
		* & \dots & \dots & *\\
		* & \dots & \dots & *
		\end{array}}_{I_{\gamma}}
	\begin{array}{c}  
	~\leftarrow ac\text{-row} \\
	~\leftarrow cd\text{-row} \\
	~\leftarrow ad\text{-row} 
	\end{array}
	$$
	
	Note that if two $\ast$ lie in the same column, they take the same value. By using Equation \eqref{eq:rowdiff}, we also see that $acd$ is in case \labelcref{case:V}. Therefore $|I_{\beta}|\leq 1$. If $|I_{\beta}|=1$, we are in case \labelcref{case:V}(\labelcref{case:Vii}), so exactly one column from $I_{\gamma}$ is zero on the $ad$-row. Using $|I_{\gamma}|=(n-4)/2$, this means that $p_{11}=(n-4)/2+2(n-4)=5n/2-10$. But section \ref{sec:caseVii} shows that any partition in which there is a vertex in case \labelcref{case:V}(\labelcref{case:Vii}) must have $p_{11}=5n/2-13$, giving a contradiction.
	
	Therefore $I_{\beta}=\emptyset$ and $acd$ is in case \labelcref{case:V}(\labelcref{case:Vi}). We also have $p_{11}=5n/2-7$, and can deduce the values for the remaining entries of the quotient matrix from Lemma \ref{lem:theta}.
\end{proof}

Therefore we can assume we have the following.
$$
\Gamma(\hat{a}bc):
\underbrace{\begin{array}{cccc}
	1 & \dots & \dots & 1\\
	1 & \dots & \dots & 1\\
	1 & \dots & \dots & 1
	\end{array}}_{I_{\omega}}
\underbrace{\begin{array}{cccc}
	1 & \dots & \dots & 1\\
	1 & \dots & \dots & 1\\
	0 & \dots & \dots & 0
	\end{array}}_{I_{\gamma}}
\begin{array}{c}  
~\leftarrow \hat{a}b\text{-row} \\
~\leftarrow \hat{a}c\text{-row} \\
~\leftarrow bc\text{-row} 
\end{array}
$$
Note that we have been careful to pick $\hat{a}$ as the index in $\hat{a}bc$ such that $\overline{bc\ast}$ is minimum. For the moment, we will let 
\begin{eqnarray*}
	W&=&I_{\omega}\cup \{\hat{a}\},\\U&=&I_{\gamma}\cup \{b,c\},
\end{eqnarray*}
and $M$ be any matching of $U,W$. For $d\in U$ we denote by $\hat{d}\in W$, its matched element.

\begin{lemma}
	We have the following.	
	\begin{enumerate}
		\item For $d\in I_{\gamma}\cup \{c\}$ we have $\overline{\hat{a}bd}=1$ and nb-array
		$$
		\Gamma(\hat{a}bd):
		\underbrace{\begin{array}{cccc}
			1 & \dots & \dots & 1\\
			1 & \dots & \dots & 1\\
			1 & \dots & \dots & 1
			\end{array}}_{I_\omega}
		\underbrace{\begin{array}{cccc}
			1 & \dots & \dots & 1\\
			1 & \dots & \dots & 1\\
			0 & \dots & \dots & 0
			\end{array}}_{I_{\gamma}\cup \{c\}\setminus \{d\}}
		\begin{array}{c}  
		~\leftarrow \hat{a}b\text{-row} \\
		~\leftarrow \hat{a}d\text{-row} \\
		~\leftarrow bd\text{-row} 
		\end{array}
		$$
		\item For $d\in I_{\gamma}\cup \{c\}$ we have $\overline{\hat{a}b\hat{d}}=1$ and nb-array
		$$
		\Gamma(\hat{a}b\hat{d}):
		\underbrace{\begin{array}{cccc}
			1 & \dots & \dots & 1\\
			1 & \dots & \dots & 1\\
			0 & \dots & \dots & 0
			\end{array}}_{I_\omega\setminus \{\hat{d}\}}
		\underbrace{\begin{array}{cccc}
			1 & \dots & \dots & 1\\
			1 & \dots & \dots & 1\\
			1 & \dots & \dots & 1
			\end{array}}_{I_{\gamma}\cup \{c\}}
		\begin{array}{c}  
		~\leftarrow \hat{a}b\text{-row} \\
		~\leftarrow b\hat{d}\text{-row} \\
		~\leftarrow \hat{b}\hat{d}\text{-row} 
		\end{array}
		$$
	\end{enumerate}	
\end{lemma} 
\begin{proof}
	By applying Equation \eqref{eq:rowdiff}, we see the following.
	\begin{align*}
	&\overline{bde}=\overline{cde}=1, \overline{\hat{a}de}=0  \qquad &(for\quad d,e\in I_{\omega})\\
	&\overline{bde}=\overline{cde}=0, \overline{\hat{a}de}=1  \qquad &(for\quad d,e\in I_{\gamma})\\
	&\overline{\hat{a}de}=\overline{bde}=\overline{cde} \qquad &(for\quad d\in I_{\omega},e\in I_{\gamma})
	\end{align*}
	The result follows immediately.
\end{proof}

\begin{cor}\label{cor:genVi}
	For all distinct $d,e,f\in U$, we have $\overline{\hat{d}ef}=1,\overline{\hat{d}e\hat{f}}=1$, and the following nb-arrays:
	\begin{enumerate}
		\item $$\Gamma(\hat{d}ef):
		\underbrace{\begin{array}{cccc}
			1 & \dots & \dots & 1\\
			1 & \dots & \dots & 1\\
			1 & \dots & \dots & 1
			\end{array}}_{W\setminus \{\hat{d}\}}
		\underbrace{\begin{array}{cccc}
			1 & \dots & \dots & 1\\
			1 & \dots & \dots & 1\\
			0 & \dots & \dots & 0
			\end{array}}_{U\setminus \{e,f\}}
		\begin{array}{c}  
		~\leftarrow \hat{d}e\text{-row} \\
		~\leftarrow \hat{d}f\text{-row} \\
		~\leftarrow ef\text{-row} 
		\end{array}$$
		\item  $$
		\Gamma(\hat{d}e\hat{f}):
		\underbrace{\begin{array}{cccc}
			1 & \dots & \dots & 1\\
			1 & \dots & \dots & 1\\
			0 & \dots & \dots & 0
			\end{array}}_{W\setminus \{\hat{d},\hat{f}\}}
		\underbrace{\begin{array}{cccc}
			1 & \dots & \dots & 1\\
			1 & \dots & \dots & 1\\
			1 & \dots & \dots & 1
			\end{array}}_{U\setminus \{e\}}
		\begin{array}{c}  
		~\leftarrow \hat{d}e\text{-row} \\
		~\leftarrow e\hat{f}\text{-row} \\
		~\leftarrow \hat{d}\hat{f}\text{-row} 
		\end{array}	$$	
	\end{enumerate}
\end{cor}
\begin{proof}
	This follows from the repeated application of the previous lemma. If $n>8$, we can we can always apply the lemma consecutively to $\hat{a}bc$ to attain the nb-array of $\hat{d}ef,\hat{d}e\hat{f}$.
\end{proof}

\begin{proof}[Proof of Theorem \ref{thm:Vi}]
	Any vertex not considered directly in Corrolary \ref{cor:genVi} is of the form $def,d\hat{d}e,\hat{d}\hat{e}\hat{f}$ or $d\hat{d}\hat{f}$ for distinct elements $d,e,f\in U$. The first two appear in the nb-array of $\hat{d}ef$, and the last two appear in the nb-array of $\hat{d}e\hat{f}$. From this we can see a vertex is in $X_1$ if and only if the vertex consists of 2 elements of $U$ and one of $W$, or vice-versa. This gives the partition $\Pi_1$ defined by the matching $M$, presented in Section \ref{sec:knJn3}.  
\end{proof}

\subsection{Case \labelcref{case:VI}(\labelcref{case:VIiii})} 

\begin{thm}\label{thm:VIiii}
	Let $abc$ be a vertex of $\Gamma$ with $\overline{abc}=1$ be such that we have nb-array as in \labelcref{case:VI}(\labelcref{case:VIiii}).
	
	Then the partition $X$ can be constructed from a matching on the elements of $\left[n\right]$, which corresponds to an instance of a partition $\Pi_2$.
\end{thm}

First we characterise the possible quotient matrices in this case.

\begin{lemma}
	Let $abc$ be a vertex of $\Gamma$ with $\overline{abc}=1$ be such that we have nb-array
	$$
	\Gamma(abc):
	\underbrace{\begin{array}{c}
		1\\
		0\\
		0
		\end{array}}_{d}
	\underbrace{\begin{array}{c}
		0\\
		1\\
		0
		\end{array}}_{e}
	\underbrace{\begin{array}{c}
		0\\
		0\\
		1
		\end{array}}_{f}
	\underbrace{\begin{array}{cccc}
		1 & \dots & \dots & 1\\
		1 & \dots & \dots & 1\\
		1 & \dots & \dots & 1
		\end{array}}_{I_{\omega}}
	\underbrace{\begin{array}{cccc}
		0 & \dots & \dots & 0\\
		0 & \dots & \dots & 0\\
		0 & \dots & \dots & 0
		\end{array}}_{I_{\beta}}
	\begin{array}{c}  
	~\leftarrow ab\text{-row} \\
	~\leftarrow ac\text{-row} \\
	~\leftarrow bc\text{-row} 
	\end{array}
	$$
	Then the quotient matrix for the partition is 
	\begin{equation*}
	\left(\begin{matrix}
	3(n-5) & 6\\
	2(n-4) & n-1
	\end{matrix}\right)
	\end{equation*}
\end{lemma}
\begin{proof}
	Using Equation \eqref{eq:rowdiff}, we deduce the following.
	\begin{align*}
	&\overline{ade}=1,\overline{bde}=\overline{cde}=0 \\
	&\overline{bdf}=1,\overline{adf}=\overline{cdf}=0\\
	&\overline{adg}=\overline{bdg}=1, \overline{cdg}=0 \qquad &(for \quad g\in I_{\omega}\cup I_{\beta})
	\end{align*}
	
	Now consider $abd$. We have $\overline{abd}=1$, and the following nb-array.
	
	$$
	\Gamma(abd):
	\underbrace{\begin{array}{c}
		1\\
		0\\
		0
		\end{array}}_{c}
	\underbrace{\begin{array}{c}
		0\\
		1\\
		0
		\end{array}}_{e}
	\underbrace{\begin{array}{c}
		0\\
		0\\
		1
		\end{array}}_{f}
	\underbrace{\begin{array}{cccc}
		1 & \dots & \dots & 1\\
		1 & \dots & \dots & 1\\
		1 & \dots & \dots & 1
		\end{array}}_{I_{\omega}\setminus \{c\}}
	\underbrace{\begin{array}{cccc}
		0 & \dots & \dots & 0\\
		1 & \dots & \dots & 1\\
		1 & \dots & \dots & 1
		\end{array}}_{I_{\beta}}
	\begin{array}{c}  
	~\leftarrow ab\text{-row} \\
	~\leftarrow ad\text{-row} \\
	~\leftarrow bd\text{-row} 
	\end{array}
	$$
	
	By using Equation \eqref{eq:rowdiff}, we see that 
	
	$$\overline{ab\ast}-\overline{ad\ast}=\frac{n-4}{2}(\overline{abc}+\overline{abe}+\overline{cde}-\overline{acd}-\overline{ade}-\overline{bce})=0$$.
	
	Therefore $I_{\beta}=\emptyset$ and $p_{11}=3(n-5)$. We can then use Lemma \ref{lem:theta} to deduce the remaining entries of the quotient matrix.
\end{proof}

Therefore, we can assume we have the following nb-array.
$$
\Gamma(\hat{a}bc):
\underbrace{\begin{array}{c}
	1\\
	0\\
	0
	\end{array}}_{a}
\underbrace{\begin{array}{c}
	0\\
	1\\
	0
	\end{array}}_{\hat{b}}
\underbrace{\begin{array}{c}
	0\\
	0\\
	1
	\end{array}}_{\hat{c}}
\underbrace{\begin{array}{cccc}
	1 & \dots & \dots & 1\\
	1 & \dots & \dots & 1\\
	1 & \dots & \dots & 1
	\end{array}}_{I_{\omega}}
\begin{array}{c}  
~\leftarrow \hat{a}b\text{-row} \\
~\leftarrow \hat{a}c\text{-row} \\
~\leftarrow bc\text{-row} 
\end{array}
$$
Note that we have been careful to denote by $\hat{b}$ as the index not in $I_{\omega}$ of the column such that $\hat{a}c\hat{b}=1$, and similarly for $a$ and $\hat{c}$.

Using Equation \eqref{eq:rowdiff}, we have the following.
\begin{align*}
&\overline{abd}=\overline{acd}=1,\overline{\hat{a}ad}=0 \qquad &(for \quad d\in I_{\omega})\\
&\overline{\hat{a}\hat{b}d}=\overline{c\hat{b}d}=1,\overline{b\hat{b}d}=0 \qquad &(for \quad d\in I_{\omega})\\
&\overline{\hat{a}\hat{c}d}=\overline{b\hat{c}d}=1,\overline{c\hat{c}d}=0 \qquad &(for \quad d\in I_{\omega})\\
&\overline{ade}=\overline{bde}=\overline{cde} \qquad &(for \quad d,e\in I_{\omega}) 
\end{align*}

Now fix $d\in I_{\omega}$. We have $\overline{\hat{a}bd}=1$ and the following nb-array.

$$
\Gamma(\hat{a}bd):
\underbrace{\begin{array}{c}
	0\\
	0\\
	1
	\end{array}}_{a}
\underbrace{\begin{array}{c}
	0\\
	1\\
	0
	\end{array}}_{\hat{b}}
\underbrace{\begin{array}{c}
	1\\
	1\\
	1
	\end{array}}_{\hat{c}}
\underbrace{\begin{array}{c}
	1\\
	1\\
	1
	\end{array}}_{c}
\underbrace{\begin{array}{cccc}
	1 & \dots & \dots & 1\\
	* & \dots & \dots & *\\
	* & \dots & \dots & *
	\end{array}}_{I_{\omega}\setminus \{d\}}
\begin{array}{c}  
~\leftarrow \hat{a}b\text{-row} \\
~\leftarrow \hat{a}d\text{-row} \\
~\leftarrow bd\text{-row} 
\end{array}
$$

Note that any two $\ast$ in the same column must be of the same value. As $p_{11}=3(n-5)$, there must be a unique $\hat{d}\in I_{\omega}$ such that $\overline{\hat{a}d\hat{d}}=0=\overline{bd\hat{d}}$. If there were another $e\in I_{\omega}\setminus \{d\}$ such that $\overline{\hat{a}e\hat{d}}=0=\overline{be\hat{d}}$, then $\overline{\hat{a}b\hat{d}}=1$, and for which the $d$ and $e$ columns are equal and contain exactly 1 row each,  giving a contradiction. Therefore we have a perfect matching of the set $I_{\omega}$, and a matching of $\left[n\right]$ given by 
\begin{equation*}
M=\{(a,\hat{a}),(b,\hat{b}),(c,\hat{c})\}\cup\{(d,\hat{d}):d\in I_{\omega}\}.
\end{equation*}
Without loss of generality, we choose a specific partition of $\left[n\right]$ into parts $U,W$, such that $M$ is a matching of $U$ with $W$ as follows.
\begin{align*}
U&=\{d:(d,\hat{d})\in M\},\\
W&=\{\hat{d}:(d,\hat{d})\in M\}.
\end{align*}

\begin{lemma}\label{lem:cVIiiipr}
	We have the following.	
	\begin{enumerate}
		\item For $d\in I_{\gamma}\cup \{c\}$ we have $\overline{\hat{a}bd}=1$ and nb-array	
		$$\Gamma(\hat{a}bd):
		\underbrace{\begin{array}{c}
			0\\
			0\\
			1
			\end{array}}_{a}
		\underbrace{\begin{array}{c}
			0\\
			1\\
			0
			\end{array}}_{\hat{b}}
		\underbrace{\begin{array}{c}
			1\\
			0\\
			0
			\end{array}}_{\hat{d}}
		\underbrace{\begin{array}{cccc}
			1 & \dots & \dots & 1\\
			1 & \dots & \dots & 1\\
			1 & \dots & \dots & 1
			\end{array}}_{I_{\omega}\setminus \{d,\hat{d}\}}
		\begin{array}{c}  
		~\leftarrow \hat{a}b\text{-row} \\
		~\leftarrow \hat{a}d\text{-row} \\
		~\leftarrow bd\text{-row} 
		\end{array}$$
		\item For $d\in I_{\gamma}\cup \{c\}$ we have $\overline{\hat{a}b\hat{d}}=1$ and nb-array
		$$\Gamma(\hat{a}b\hat{d}):
		\underbrace{\begin{array}{c}
			0\\
			0\\
			1
			\end{array}}_{\hat{b}}
		\underbrace{\begin{array}{c}
			0\\
			1\\
			0
			\end{array}}_{a}
		\underbrace{\begin{array}{c}
			1\\
			0\\
			0
			\end{array}}_{\hat{d}}
		\underbrace{\begin{array}{cccc}
			1 & \dots & \dots & 1\\
			1 & \dots & \dots & 1\\
			1 & \dots & \dots & 1
			\end{array}}_{I_{\omega}\setminus \{d,\hat{d}\}}
		\begin{array}{c}  
		~\leftarrow \hat{a}b\text{-row} \\
		~\leftarrow b\hat{d}\text{-row} \\
		~\leftarrow \hat{a}\hat{d}\text{-row} 
		\end{array}$$
	\end{enumerate}
\end{lemma} 
\begin{proof}
	This follows immediately from the above discussion.
\end{proof}

\begin{cor}\label{cor:genVIiii}
	For all distinct $d,e,f\in U$, we have $\overline{\hat{d}ef}=1,\overline{\hat{d}e\hat{f}}=1$, and the following nb-arrays:
	\begin{enumerate}
		\item $$\Gamma(\hat{d}ef):
		\underbrace{\begin{array}{c}
			0\\
			0\\
			1
			\end{array}}_{d}
		\underbrace{\begin{array}{c}
			0\\
			1\\
			0
			\end{array}}_{\hat{e}}
		\underbrace{\begin{array}{c}
			1\\
			0\\
			0
			\end{array}}_{\hat{f}}
		\underbrace{\begin{array}{cccc}
			1 & \dots & \dots & 1\\
			1 & \dots & \dots & 1\\
			1 & \dots & \dots & 1
			\end{array}}_{(U\setminus \{d,e,f\})\cup(W\setminus \{\hat{d},\hat{e},\hat{f}\})}
		\begin{array}{c}  
		~\leftarrow \hat{d}e\text{-row} \\
		~\leftarrow \hat{d}f\text{-row} \\
		~\leftarrow ef\text{-row} 
		\end{array}$$
		\item $$\Gamma(\hat{d}e\hat{f}):
		\underbrace{\begin{array}{c}
			0\\
			0\\
			1
			\end{array}}_{\hat{e}}
		\underbrace{\begin{array}{c}
			0\\
			1\\
			0
			\end{array}}_{d}
		\underbrace{\begin{array}{c}
			1\\
			0\\
			0
			\end{array}}_{f}
		\underbrace{\begin{array}{cccc}
			1 & \dots & \dots & 1\\
			1 & \dots & \dots & 1\\
			1 & \dots & \dots & 1
			\end{array}}_{(U\setminus \{d,e,f\})\cup (W\setminus \{\hat{d},\hat{e},\hat{f}\})}
		\begin{array}{c}  
		~\leftarrow \hat{d}e\text{-row} \\
		~\leftarrow e\hat{f}\text{-row} \\
		~\leftarrow \hat{d}\hat{f}\text{-row} 
		\end{array}$$
		
	\end{enumerate}
\end{cor}
\begin{proof}
	This follows from the repeated application of Lemma \ref{lem:cVIiiipr}. If $n>8$, we can we can always apply the lemma consecutively to $\hat{a}bc$ to attain the nb-array of $\hat{d}ef,\hat{d}e\hat{f}$.
\end{proof}

\begin{proof}[Proof of Theorem \ref{thm:VIiii}]
	Any vertex not considered directly in Corrolary \ref{cor:genVIiii} is of the form $def,d\hat{d}e,\hat{d}\hat{e}\hat{f}$ or $d\hat{d}\hat{f}$ for distinct elements $d,e,f\in U$. The first two appear in the nb-array of $\hat{d}ef$, and the last two appear in the nb-array of $\hat{d}e\hat{f}$. From this we can see a vertex is in $X_1$ if and only if the vertex does not contain a pair $d\hat{d}$ such that $d\in U$. This gives the partition $\Pi_2$ defined by the $M$ presented in Section \ref{sec:knJn3}.  
\end{proof}

\section*{Acknowledgements}
R. J. Evans was supported by the Mathematical
Center in Akademgorodok, under agreement No. 075-15-2019-1613 with the Ministry
of Science and Higher Education of the Russian Federation.

\bibliographystyle{plain}
\bibliography{references.bib}
	
\end{document}